\documentclass[10pt]{amsart}
\usepackage{amssymb, mathrsfs, amsthm, amsmath}
\usepackage[all, cmtip]{xy}

\title[The filled Julia set of a Drinfeld module]{The filled Julia set of a Drinfeld module and uniform bounds for torsion}\author{Patrick Ingram}
\email{pingram@math.colostate.edu}
\address{Department of Mathematics, Colorado State University, Fort Collins, USA}

\newcommand{\QQ}{\mathbb{Q}}
\newcommand{\ZZ}{\mathbb{Z}}
\newcommand{\CC}{\mathbb{C}}
\newcommand{\RR}{\mathbb{R}}
\newcommand{\FF}{\mathbb{F}}
\newcommand{\PP}{\mathbb{P}}
\renewcommand{\AA}{\mathbb{A}}
\newcommand{\Ocal}{\mathcal{O}}
\newcommand{\Gal}{\operatorname{Gal}}

\newcommand{\Aberk}[1]{\mathbb{A}^1_{\mathrm{Berk}}}
\newcommand{\Pberk}[1]{\mathbb{P}^1_{\mathrm{Berk}}}

\newcommand{\fj}[1]{{#1}^\mathrm{FJ}}

\renewcommand{\epsilon}{\varepsilon}

\newtheorem{theorem}{Theorem}[section]
\newtheorem{proposition}[theorem]{Proposition}
\newtheorem{lemma}[theorem]{Lemma}
\newtheorem{conjecture}[theorem]{Conjecture}

\theoremstyle{remark}
\newtheorem{remark}[theorem]{Remark}

\theoremstyle{definition}
\newtheorem{definition}[theorem]{Definition}

\begin{document}
\maketitle
\begin{abstract}
If $\phi$ is a Drinfeld module over a local function field $L$, we may view $\phi$ as a dynamical system, and consider its filled Julia set $\fj{\phi}(L)$.  If $\phi^0(L)$ is the connected component of the identity, relative to the Berkovich topology, we give a characterisation of the component module $\fj{\phi}(L)/\phi^0(L)$ which is analogous to the Kodaira-N\'{e}ron characterisation of the special fibre of a N\'{e}ron model of an elliptic curve over a non-archimedean field.  In particular, if $L$ is the fraction field of a discrete valuation ring, then the component module is finite, and moreover trivial in the case of good reduction.

In the context of global function fields, the filled Julia set may be considered as an object over the ring of finite adeles.  In this setting we formulate a conjecture about the structure of the (finite) component module which, if true, would imply Poonen's Uniform Boundedness Conjecture for torsion on Drinfeld modules of a given rank over a given global function field.  Finally, we prove this conjecture for certain families of Drinfeld modules, obtaining uniform bounds on torsion in some special cases.
\end{abstract}

\section{Introduction}
\label{sec:intro}

Let $F/\QQ$ be a number field, and let $E/F$ be an elliptic curve.  It follows from the Mordell-Weil Theorem that the torsion subgroup $E^{\mathrm{Tors}}(F)$ of the group of $F$-rational points on $E$ is finite, but something much stronger is true.  A well-known result of Mazur \cite{mazur}, in the case $F=\QQ$, and of Merel \cite{merel}, more generally, ensures that the size of this group is bounded uniformly as $E/F$ varies.  Indeed, Merel's Theorem states that this bound depends not on the number field $F$, but only on the degree $[F:\QQ]$.

Now let $K$ be the function field of a curve $X$ over a finite field $\FF_q$, and let $A$ be the ring of regular functions at some point $\infty\in X(\FF_q)$.  If $L/K$ is a finite extension, then a Drinfeld $A$-module over $L$ is a ring homomorphism
\[\phi:A\to\operatorname{End}_L(\mathbb{G}_\mathrm{a})\]
satisfying certain additional conditions (see Section~\ref{sec:prelims} for more details). We write $\phi(L)$ for the corresponding $A$-module structure on $\mathbb{G}_\mathrm{a}(L)$.  Although $\phi(L)$ is never finitely generated, it is a theorem of Poonen~\cite{poonen} that $\phi(L)$ is tame, and hence the torsion submodule $\phi^{\mathrm{Tors}}(L)$ is finite.  It is natural to ask how uniformly one might bound its size.
\begin{conjecture}[Poonen~\cite{poonenunif}]\label{conj:poonenconj}
Let $L/K$ be a finite extension.  Then there is a bound on the quantity $\#\phi^{\mathrm{Tors}}(L)$, as $\phi/L$ varies over Drinfeld $A$-modules of a given rank.
\end{conjecture}
Results in the direction of Conjecture~\ref{conj:poonenconj} are scarce, although Poonen has proven the statement for twists of a fixed Drinfeld module~\cite{poonenunif},  and Armana~\cite{armana} and P\'{a}l~\cite{pal} have proven general results for Drinfeld modules of rank 2, subject to some additional constraints, using deep methods related to the approach used successfully for elliptic curves.

The purpose of the present paper is to note that, if one views Drinfeld modules as arithmetic dynamical systems, then Conjecture~\ref{conj:poonenconj} follows from a conjecture about the structure of the adelic filled Julia sets.  More concretely, for each place $v$ of $L$, let $L_v$ denote the completion of $L$ at $v$, and let $\CC_v$ be the completion of an algebraic closure of $L_v$.  Note that the sum of two $v$-adically bounded submodules of $\phi(\CC_v)$ is again bounded, and so there is a unique maximal bounded submodule of $\phi(\CC_v)$, which we call the \emph{filled Julia set} $\fj{\phi}(\CC_v)$ of $\phi(\CC_v)$.  If $\Aberk{\CC_v}$ denotes the Berkovich analytic space associated to $\CC_v$, then there is a natural embedding $\CC_v\hookrightarrow \Aberk{\CC_v}$, and we denote by $\phi^0(\CC_v)$ the set of $\CC_v$-rational points in the connected component of $\overline{\fj{\phi}(\CC_v)}$ containing the identity $0$, where $\overline{X}$ is the closure of $X\subseteq \Aberk{\CC_v}$ relative to the Berkovich topology.
  Similarly let $\fj{\phi}(L_v)$ and $\phi^0(L_v)$ denote the $L_v$-rational points in $\fj{\phi}(\CC_v)$ and $\phi^0(\CC_v)$.  It turns out that the quotient $\fj{\phi}(L_v)/\phi^0(L_v)$ has the structure of a finite $A$-module (see Theorem~\ref{th:drinfeldKN} below).  Our methods are strongly inspired by earlier work of Ghioca~\cite{ghioca1},  although it appears that the central object of study, the component module $\fj{\phi}(L_v)/\phi^0(L_v)$, has not previously appeared in the literature.

Given a Drinfeld $A$-module $\phi$ over $L$, we will define in Section~\ref{sec:prelims} a point $j_\phi$ in a certain weighted projective space such that $\phi$ is $\overline{L}$-isomorphic to $\psi$  if and only if $j_\phi=j_\psi$.  In the case of a Drinfeld $\FF_q[T]$-module $\phi$, with
\[\phi_T(x)=Tx+a_1x^q+\cdots +a_rx^{q^r},\]
we will simply take
\[j_\phi=[a_1:a_2:\cdots:a_r],\]
where the $i$th entry has weight $q^i-1$.  We will define below a quantity $0\leq \mu(\phi, N, \mathfrak{a})\leq 1$, for any integer $N\geq 0$ and any ideal $\mathfrak{a}\subseteq A$.  Although we refer the reader to Definition~\ref{def:mudef} for the precise description, 
 the inequality $\mu(\phi, N, \mathfrak{a})\geq \epsilon$ may be informally described as follows.  The point $j_\phi$ is a point in a weighted projective space, and so has a height $h(j_\phi)$ consisting of contributions from finite and infinite places of the field of definition. The inequality $\mu(\phi, N, \mathfrak{a})\geq \epsilon$ asserts that at least proportion $\epsilon$ of the finite part of the height comes from places where the component module $\fj{\phi}(L_v)/\phi^0(L_v)$ is ``small'', annihilated by $\mathfrak{a}$, and up to $N$ other places.

\begin{theorem}\label{th:szpiroresult}
Fix $A$ and $L$ as above, let $N\geq 0$, and let $\mathfrak{a}\subseteq A$ be any proper ideal.  If $\phi/L$ is a Drinfeld module of rank $r$ satisfying $\mu(\phi, N, \mathfrak{a})\geq 1/q$,  then $\# \phi^{\mathrm{Tors}}(L)$ is bounded by a constant which depends only on $N$, $\mathfrak{a}$, and $r$.
\end{theorem}

We will see below that, for any given $\phi/L$, we may take $N$ to be the number of places at which $\phi$ has bad reduction
 and thereby obtain $\mu(\phi, N, \mathfrak{a})=1$.  Theorem~\ref{th:szpiroresult} thus contains a theorem of Ghioca~\cite{ghioca1} confirming Conjecture~\ref{conj:poonenconj} for Drinfeld modules with bad reduction at a bounded number of places.  The strength in Theorem~\ref{th:szpiroresult} over the earlier result is that, while one may easily construct Drinfeld modules with arbitrarily many places of bad reduction, thereby showing that the results in \cite{ghioca1} are strictly weaker than the claim in Conjecture~\ref{conj:poonenconj}, it seems at least plausible that the conditions of  Theorem~\ref{th:szpiroresult} are met with some uniformity.  In particular, an affirmative answer to the following conjecture would yield a proof of Conjecture~\ref{conj:poonenconj}.

\begin{conjecture}\label{question}
Fix a finite extension $L/K$.  Then there exists an $N\geq 0$ and an $\mathfrak{a}\subseteq A$ such that $\mu(\phi, N, a)\geq 1/q$ for every Drinfeld $A$-module $\phi/L$.
\end{conjecture}
Although we will explain below how Conjecture~\ref{question} is motivated by Szpiro's Conjecture for elliptic curves, perhaps the best evidence for Conjecture~\ref{question} is that we can prove it for certain families of Drinfeld modules.
Consider some curve $C/L$, and a Drinfeld module $\phi/L(C)$.  If the coefficients of $\phi$ are regular at $\beta\in C(L)$, we may specialize to obtain a Drinfeld module $\phi_\beta/L$.  If $C$ is defined over $\FF_q\subseteq L$ and the coefficients of $\phi_a(x)-ax\in L(C)[x^q]$ happen to land in $\FF_q(C)$, then we will call this family a \emph{simple family}.  As an example of a simple family of Drinfeld $\FF_q[T]$-modules over $\PP^1$, consider the family over any extension $L/K$ defined by
\[\phi_{\beta, T}(x)=Tx+x^q+(\beta^3-\beta)x^{q^r}\]
for $\beta\in\PP^1(L)$.
Finally, we let $m_2=5$, $m_3=4$, and $m_q=3$ for $q\geq 4$.
\begin{theorem}\label{th:simplefams}
Suppose $A=\FF_q[T]$, let $L/K$ be a rational extension which is at most quadratic, let $\phi/L(\PP^1)$ be a simple family of Drinfeld modules, such that the generic fibre has at least $m_q$ geometric places of genuinely bad (i.e., not potentially good) reduction.  Then the family of fibres $\phi_\beta$ with $\beta\in \PP^1(L)$ satisfies Conjecture~\ref{question}, for $N=0$ and some ideal $\mathfrak{a}\subseteq A$ which is independent of $\beta$.  In particular, $\# \phi_\beta^\mathrm{Tors}(L)$ is bounded uniformly for $\beta\in \PP^1(L)$.  
\end{theorem}

Theorem~\ref{th:szpiroresult} justifies some curiosity as to the structure of the component module $\fj{\phi}(L_v)/\phi^0(L_v)$. 
 If $F_\mathfrak{p}/\QQ_p$ is a finite extension, if $E/F_\mathfrak{p}$ is an elliptic curve with split multiplicative reduction, and if $E^0(F_\mathfrak{p})$ is the connected component of the identity, then 
\[E(F_\mathfrak{p})/E^0(F_\mathfrak{p})\cong \ZZ/N\ZZ,\] for $N=-v(j_E)$, by well-known work of Kodaira, N\'{e}ron, and Tate (see, e.g., \cite{ataec}).
If $\CC_p$ is a completion of the algebraic closure of $F_\mathfrak{p}$, we have \[E(\CC_p)/E^0(\CC_p)\cong \QQ/\ZZ.\]  Replacing $\ZZ$ with $A$, $\QQ$ with $K$, and $E$ with $\phi$, we see a strikingly similar structure to $\fj{\phi}(L_v)/\phi^0(L_v)$.
\begin{theorem}\label{th:drinfeldKN}
Suppose $A=\FF_q[T]$, let $L_v/K_v$ be a finite extension, let $\CC_v$ be the completion of the algebraic closure of $L_v$, and let $\phi/L_v$ be a Drinfeld module of rank $r$ with potentially stable reduction of rank $r-s$.  Then there exist $a_1, ..., a_s\in A$, of degree bounded in terms of $s$ and $-v(j_\phi)$, such that
\[\fj{\phi}(L_v)/\phi^0(L_v)\cong \bigoplus_{i=1}^s A/a_iA,\]
and
\[\fj{\phi}(\CC_v)/\phi^0(\CC_v)\cong (K/A)^{s}.\]
In particular, both modules are trivial if $\phi$ has potentially good reduction.
\end{theorem}

Extending the analogy with elliptic curves further, we may define a quantity $0\leq \mu(E, N, \mathfrak{a})\leq 1$, for any $N\geq 0$ and ideal $\mathfrak{a}\subseteq \ZZ$, which represents the proportion of the height of $j_E$ coming from  places where the component group $E(F_\mathfrak{p})/E^0(F_\mathfrak{p})$ is annihilated by $\mathfrak{a}$, ignoring the contributions of archimedean places and at most $N$ other places (which will be made precise in Section~\ref{sec:elliptic}).
We recall a weak form of Szpiro's Conjecture for semistable elliptic curves.
\begin{conjecture}[Szpiro \cite{szpiro}]
There exists a constant $\sigma$, depending only on $F$, such that if $E/F$ is a semistable elliptic curve with minimal discriminant $\Delta_E$ and conduction $f_E$, then
\[\frac{\log|\mathrm{Norm}_{K/\QQ}\Delta_E|}{\log|\mathrm{Norm}_{K/\QQ}f_E|}\leq \sigma.\]
\end{conjecture}
The ratio on the left is commonly known as the Szpiro ratio $\sigma(E/F)$, and at least for $F=\QQ$ it is conjectured that we may take any $\sigma>6$, if we allow finitely many exceptional $E/\QQ$.  The following result is essentially due to Hindry and Silverman~\cite{hindry-silv}, although we consider only semistable elliptic curves for simplicity.
\begin{theorem}\label{th:elliptic}
Let $E/F$ be semistable, let $n\geq 1$, and let $\mathfrak{a}=(n!)\subseteq\ZZ$.  Then \[\mu(E, 0, \mathfrak{a})\geq \frac{1-\sigma(E/F)/n}{\sigma(E/F)(1-1/n)}.\]  In particular, if Szpiro's Conjecture holds for the field $K$, and $\sigma(E, N)\leq \sigma$ for all semistable $E/F$, then for any $\epsilon>0$ there is an ideal $\mathfrak{a}\subseteq \ZZ$ such that \[\mu(E, 0,  \mathfrak{a})\geq 1/\sigma-\epsilon\] for all semistable $E/F$.
\end{theorem}
Considering the quantity $\mu(E, N, \mathfrak{a})$ for $N>0$ brings  us into the territory of Silverman's ``prime-depleted'' version of Szpiro's Conjecture~\cite{depszpiro}.

In Section~\ref{sec:prelims} we establish some notation, and recall the basics of Drinfeld modules.  In Section~\ref{sec:localfields} we consider Drinfeld modules over local fields, and establish some basic results on the structure of the filled Julia set, most notably the finiteness of the component module at finite places.  In Section~\ref{sec:global}, we consider the filled Julia set  in the context of global function fields, and we prove Theorems~\ref{th:szpiroresult} and~\ref{th:simplefams}.  Finally, in Section~\ref{sec:uniformization} we turn to a more detailed examination of the case where $A=\FF_q[T]$, and compare the structure of the component module to the analogous object in the case of elliptic curves.

%%%%%%%%%%%%%%%%%%%%%%%%%%%%%%%%%%%%%%%%%%%%%%%%%
%%%%%%%%%%%%%%%%%%%%%%%%%%%%%%%%%%%%%%%%%%%%%%%%%
%%%%%%%%%%%%%%%%%%%%%%%%%%%%%%%%%%%%%%%%%%%%%%%%%
%%%%%%%%%%%%%%%%%%%%%%%%%%%%%%%%%%%%%%%%%%%%%%%%%
%%%%%%%%%%%%%%%%%%%%%%%%%%%%%%%%%%%%%%%%%%%%%%%%%
%%%%%%%%%%%%%%%%%%%%%%%%%%%%%%%%%%%%%%%%%%%%%%%%%

\section{Notation and preliminaries}
\label{sec:prelims}

\subsection{Drinfeld modules}
Throughout, we suppose that $q$ is a power of a prime, and that $K$ is a function field in one variable over $\FF_q$, that is, that $K=\FF_q(X)$ for some algebraic curve $X/\FF_q$.  We fix a place $\infty\in X(\FF_q)$, and let $A\subseteq K$ denote the ring of regular functions at $\infty$.  If $a\in A$, then $\deg(a)$ will denote the order of the pole of $a$ at $\infty$, and we set $|a|_\infty=q^{\deg(a)}$; note that this agrees with our definitions below for the absolute value on $K$ corresponding to the point $\infty$.  By an \emph{$A$-field}, we mean a field $L$ with a homomorphism $i:A\to L$, and we will consider only the case in which $L$ has \emph{generic characteristic}, that is, where $i$ is an injection; the typical example is where $L/K$ is a finite extension, and $i$ is the inclusion map.

If $L$ is an $A$-field, then a \emph{Drinfeld} $A$-module over $L$ is a homomorphism 
\begin{gather*}\phi:A\to \mathrm{End}_L(\mathbb{G}_\mathrm{a})\\
a\mapsto \phi_a,
\end{gather*}
 with the property that, for all $a\in A$,
\[\phi_a(x)=ax+O(x^q)\in L[x],\]
but $\phi_a(x)\neq ax$ for at least some $a\in A$.  
Drinfeld~\cite{drinfeld} proved the existence of an integer $r\geq 1$ such that $\deg(\phi_a(x))=|a|_\infty^r$, for all $a\in A$, and this quantity $r$ will be known as the \emph{rank} of $\phi$.  

Two Drinfeld modules $\phi/L$ and $\psi/L$ are said to be \emph{isomorphic} over the extension $L'/L$ if and only if there exists an $\alpha\in L'$ such that $\phi_a(\alpha x)=\alpha\psi_a(x)$ for all $a\in A$, abbreviated $\phi\alpha=\alpha\psi$.  Suppose that we have fixed a ordered set of generators $\{T_1, ..., T_m\}$ for $A$ as an $\FF_q$-algebra.  Then we have, for each $i$,
\[\phi_{T_i}(x)=T_i+a_{i, 1}x^q+\cdots +a_{i, \deg(T_i)r}x^{q^{\deg(T_i)r}},\]
with $a_{i, j}\in L$.  By the \emph{$\vec{w}$-weighted projective space} $\PP^{\vec{w}}$, where $\vec{w}=(w_1, ..., w_{m+1})\in (\ZZ^+)^{m+1}$, we mean the quotient of $\AA^{m+1}\setminus\{(0, 0, ..., 0)\}$ under the $\mathbb{G}_\mathrm{m}(\overline{L})$ action
\[(x_1, x_2, ..., x_{m+1})\to(\alpha^{w_1}x_1, \alpha^{w_2}x_2, ..., \alpha^{w_{m+1}}x_{m+1}).\]
We warn the reader that, in general, points of $\PP^{\vec{w}}$ which are fixed by $\Gal(\overline{L}/L)$ do not necessarily have a representative with coordinates in $L$, unlike in the case of the usual projective space.

\begin{definition}
Fix a set of generators $T_1, ..., T_m$ for $A$ as an $\FF_q$-algebra, fix $r\geq 1$, and let
 $M_{A, r}$ denote the weighted projective space with coordinates $x_{i, j}$, for $1\leq i\leq m$ and $1\leq j\leq r\deg(T_i)$, such that $x_{i, j}$ is given weight $q^j-1$.  If $\phi/L$ is a Drinfeld $A$-module of rank $r$, then by the \emph{$j$-invariant} of $\phi/L$, we mean the point
\[j_\phi=[a_{1, 1}, a_{1, 2}, ..., a_{1, \deg(T_1)r}, a_{2, 1}, ..., a_{2, \deg(T_2)r}, ..., a_{m, 1}, a_{m, 2}, ..., a_{m, \deg(T_m)r}]\] in $M_{A, r}(L)$.
 \end{definition}

We note the following result, which is proven in \cite{jpaper}.
\begin{lemma}
Let $\phi/L$ and $\psi/L$ be two Drinfeld $A$-modules.  Then the following are equivalent:
\begin{enumerate}
\item $\phi$ and $\psi$ are $L^{\mathrm{sep}}$-isomorphic;
\item $\phi$ and $\psi$ are $\overline{L}$-isomorphic;
\item $j_\phi=j_\psi$.
\end{enumerate}
\end{lemma}

\subsection{Heights and valuations}

Throughout, we will make the convention of normalizing logarithms so that $\log q=1$, and write $\log^+ x$ for $\log\max\{1, x\}$.
To each prime ideal $\mathfrak{p}\subseteq A$, we associate a normalized valuation $v$, and set $|x|_v=q^{-v(x)\deg(v)}$ where, as usual, $\deg(v)=[A/\mathfrak{p} A:\FF_q]$, or equivalently, the number of points in the $\Gal(\overline{\FF_q}/\FF_q)$-orbit in $X(\FF_q)$ corresponding to $\mathfrak{p}$.  These valuations on $K$ form the set $M_K^0$ of \emph{finite places} of $K$, and the remaining absolute value $|\cdot|_\infty$ is the infinite place, the sole member of $M_K^\infty$.

If $L/K$ is a finite extension, we denote by $M_L$ the set of valuations extending the places in $M_K$, normalized so that if $w\in M_K$ is the place below $v\in M_L$, then  $|x|_v=|x|_w$ for all $x\in K$.  In other words, we take \[|x|_v=q^{-v(x)\deg(v)/[L_v:K_v]},\]
and note that this agrees with $|x|_w$ on $K$.
  We note two important facts, first that
\begin{equation}\label{eq:extensions}\sum_{v\mid w}[L_v:K_v]=[L:K]\end{equation}
for any $w\in M_K$, where $v\mid w$ means that $v$ is a place extending $w$ (see, e.g., \cite{localfields}), and
\begin{equation}\label{eq:productfla}\sum_{v\in M_L}[L_v:K_v]\log|x|_v=0,\end{equation}
for any non-zero $x\in L$.

For any $x\in L$, we will define the \emph{height} of $x$ to be
\[h(x)=\sum_{v\in M_L}\frac{[L_v:K_v]}{[L:K]}\log^+|x|_v.\]
Note that, by~\eqref{eq:extensions}, $h$ extends to a well-defined function $h:\overline{K}\to \RR$.

 More generally, if $\PP^{\vec{w}}$ is the weighted projective space with weights $\vec{w}=(w_0, ..., w_N)$ (all non-zero), then we define a height on $\PP^{\vec{w}}$ by
\[h([x_0:\cdots:x_N])=\sum_{v\in M_L}\frac{[L_v:K_v]}{[L:K]}\log\max\{|x_0|_v^{1/w_0},\cdots |x_1|_v^{1/w_N}\}.\]
Note that this is well-defined, by~\eqref{eq:productfla}, and satisfies the Northcott property, which is easily shown from the Northcott property of $\PP^N$ by examining the morphism $\Phi:\PP^{\vec{w}}\to \PP^{N}$ given by
\[\Phi([x_0:x_1:\cdots :x_N])=[x_0^{w_1w_2\cdots w_N}:x_1^{w_0w_2\cdots w_N}:\cdots :x_N^{w_0w_1\cdots w_{N-1}}].\]
(In particular, one notes that $h(\Phi(x))=\left(\prod w_i\right) h(x)$.)  In particular, for any given $A$ and $r$, and values $B_1$ and $B_2$, there are only finitely many points $j\in M_{A, r}(L^{\mathrm{sep}})$ with $h(j)\leq B_1$ and $[L(j):L]\leq B_2$.  In particular, there are only finitely many $L$-isomorphism classes of Drinfeld module $\phi/L$ with $h(j_\phi)\leq B_1$.

%%%%%%%%%%%%%%%%%%%%%%%%%%%%%%%%%%%%%%%%%%%%%%%%%
%%%%%%%%%%%%%%%%%%%%%%%%%%%%%%%%%%%%%%%%%%%%%%%%%
%%%%%%%%%%%%%%%%%%%%%%%%%%%%%%%%%%%%%%%%%%%%%%%%%
%%%%%%%%%%%%%%%%%%%%%%%%%%%%%%%%%%%%%%%%%%%%%%%%%
%%%%%%%%%%%%%%%%%%%%%%%%%%%%%%%%%%%%%%%%%%%%%%%%%
%%%%%%%%%%%%%%%%%%%%%%%%%%%%%%%%%%%%%%%%%%%%%%%%%

\section{Drinfeld modules over local fields}
\label{sec:localfields}

Let $\CC_v$ be any complete, algebraically closed $A$-field containing $K$, and let $\phi/\CC_v$ be a Drinfeld module.  We will denote by $\Aberk{\CC_v}$ the Berkovich analytic space associated to $\AA^1_{\CC_v}$, as described in \cite{baker-rumely}.  Briefly, this is the space of multiplicative seminorms on the ring $\CC_v[x]$ which extend the absolute value on $\CC_v$, appropriately topologized.  Every point $\zeta\in \CC_v$ may be identified with a (unique) seminorm $\|f\|_\zeta=|f(\zeta)|$, giving an embedding $\AA^1_{\CC_v}\to\Aberk{\CC_v}$; the former turns out to be dense in the latter. If $\|\cdot\|_\zeta$ is the multiplicative semi-norm associated to $\zeta\in\Aberk{\CC_v}$, then $a\cdot \zeta$ is defined by
\[\|f\|_{a\cdot \zeta}=\|f\circ\phi_a\|_\zeta.\]
It is shown in \cite{baker-rumely} that this gives a continuous map $\phi_a:\Aberk{\CC_v}\to \Aberk{\CC_v}$, and yields a continuous $A$-action on $\Aberk{\CC_v}$, where $A$ is considered with the discrete topology.  This extends the $A$-action $x\mapsto \phi_a(x)$ on the classical points $\AA^1_{\CC_v}\subseteq \Aberk{\CC_v}$.

Write $\Ocal_A(\zeta)$ for the $A$-orbit of $\zeta\in \Aberk{\CC_v}$.
The \emph{Berkovich filled Julia set $\fj{\phi}_\mathrm{Berk}\subseteq\Aberk{\CC_v}$} is the set of points $\zeta$ for which $\Ocal_A(\zeta)$ is bounded.  For any subfield $F\subseteq\CC_v$, we define $\fj{\phi}(F)=\fj{\phi}_\mathrm{Berk}\cap F$.  We will write $\phi^0_\mathrm{Berk}\subseteq\fj{\phi}_\mathrm{Berk}$ for the connected component containing 0, with respect to the Berkovich topology.

Note that $\fj{\phi}(\CC_v)$ is simply the set
\begin{equation*}
\fj{\phi}(\CC_v)=\left\{x\in\CC_v:\text{there exists a }B\geq 0\text{ such that }|\phi_a(x)|\leq B\text{ for all }a\in A\right\}.
\end{equation*}
If $G_\phi$ is the Green's function associated to $\phi$, as in \cite{jpaper}, then $\fj{\phi}_\mathrm{Berk}$ may de described as the locus on which $G_\phi$ vanishes, although we will not use that fact here.

For later convenience, we note that $\fj{\phi}(\CC_v)$ can be described as the, \emph{a priori} larger, Berkovich filled Julia set of $\phi_T(x)$, for any non-constant $T\in A$.  The fact that these filled Julia sets all coincide is a fact about commuting maps well-known in the context of holomorphic dynamics.
\begin{lemma}
For any non-constant $T\in A$, $\fj{\phi}_\mathrm{Berk}$ is the Berkovich filled Julia set associated to $\phi_T(x)$.
\end{lemma}

\begin{proof}
Let $J$ denote the Berkovich filled Julia set of $\phi_T(x)$, and note that $J\supseteq \fj{\phi}_\mathrm{Berk}$ trivially, since $\Ocal_A(\zeta)$ always contains the orbit under $\phi_T$ of $\zeta$.  On the other hand, it is straightforward to show that for any $a\in A$, $\phi_a(x)$ takes bounded sets in $\Aberk{\CC_v}$ to bounded sets, and so if $\zeta\in J$, we have
\[\{\phi_a(\zeta), \phi_T(\phi_a(\zeta)), \phi_T^2(\phi_a(\zeta)), ...\}=\phi_a\left(\{\zeta, \phi_T(\zeta), \phi_T^2(\zeta), ...\}\right)\]
bounded.  It follows that $J$ is closed under the action of $A$, and so $J\subseteq \fj{\phi}_\mathrm{Berk}$.
\end{proof}

Our first proposition gives some basic algebraic information about the sets $\fj{\phi}(\CC_v)$ and $\phi^0(\CC_v)$.

\begin{proposition}
The subsets $\phi^0(\CC_v)\subseteq \fj{\phi}(\CC_v)\subseteq \CC_v$ are submodules of $\phi(\CC_v)$.
\end{proposition}

The proof of this will come in a sequence of lemmas.  We note, first of all, that it is quite simple to show that $\fj{\phi}(\CC_v)$ is a submodule of $\phi(\CC_v)$.

\begin{lemma}
The set $\fj{\phi}(\CC_v)$ is a submodule of $\phi(\CC_v)$.
\end{lemma}

\begin{proof}
If $\zeta_1\in \Ocal_A(\zeta_2)$, it is clear that $\Ocal_A(\zeta_1)\subseteq \Ocal_A(\zeta_2)$, and so $\fj{\phi}_\mathrm{Berk}$ is closed under the action of $A$.  It suffices to show, then, that $\fj{\phi}(\CC_v)$ is an additive subgroup of $\phi(\CC_v)$, which follows from the fact that $0\in \fj{\phi}(\CC_v)$, that
\[|\phi_a(x+y)|\leq \max\{|\phi_a(x)|, |\phi_a(y)|\}\]
for all $a\in A$ and $x, y\in \fj{\phi}(\CC_v)$, and the observation that $|\phi_a(-x)|=|\phi_a(x)|$.
\end{proof}

Since it will be useful later, we now present a disk of finite radius which contains $\fj{\phi}(\CC_v)$.
\begin{lemma}\label{lem:juliadisk}
Define a real number $B_T>0 $ by
\[\log B_T=\max\{\log|\xi|:\xi\in \phi[T]\}+\frac{1}{q^{r\deg(T)}-1}\log^+|T^{-1}|.\]
Then $\fj{\phi}(\CC_v)\subseteq D(0, B_T)$.
\end{lemma}

\begin{proof}
It is straightforward to see that if $|x|>B_T$, then
\[|\phi_T(x)|=|\Delta x^{q^{r\deg(T)}}|>|x|,\]
where $\Delta$ is the leading term of $\phi_T(x)$.  By induction,
\[q^{-Nr\deg(T)}\log|\phi_{T^N}(x)|=\log|x|+\frac{1-q^{-Nr\deg(T)}}{q^{r\deg(T)}-1}\log|\Delta|.\]
If $x\in\fj{\phi}(\CC_v)$, then the left-hand side tends to 0 with $N\to\infty$, but the right-hand side tends to
\begin{eqnarray*}
\log|x|+\frac{1}{q^{r\deg(T)}-1}\log|\Delta|&>&\frac{1}{q^{r\deg(T)}-1}\left(\log\left|\frac{T}{\Delta}\right|+\log^+|T^{-1}|+\log|\Delta|\right)\\
&=&\frac{1}{q^{r\deg(T)}-1}\log^+|T|\geq 0,
\end{eqnarray*}
since $\frac{1}{q^{r\deg(T)}-1}\log|T/\Delta|$ is the average of the slopes in the Newton polygon for $\phi_T(x)$, and hence a lower bound for $\max\{\log|\xi|:\xi\in \phi[T]\}$.

\end{proof}

Now that we have some description of $\fj{\phi}(\CC_v)$, we will return to the quantity $c(\phi)$ defined above.
For any polynomial $f(x)\in \CC_v[x]$ of degree $d\geq 2$, with leading coefficient $a_d$, let
\[c_v(f)=\frac{1}{d-1}\log|a_d^{-1}|_v.\]  The following is Theorem~4.1 of \cite{baker-hsia}, although we warn the reader that our $c_v(f)$ is the logarithm of theirs.
\begin{theorem}[Baker-Hsia \cite{baker-hsia}]
The (logarithmic) transfinite diameter of the filled Julia set of $f$ is $c_v(f)$.
\end{theorem}
As a corollary, since $\fj{\phi}_\mathrm{Berk}$ is the filled Julia set of any $\phi_T(x)$,  we note that $c_v(\phi_T)$ is independent of the choice of $T\in A\setminus\FF_q$.  We will write $c_v(\phi)$ for this quantity, or $c(\phi)$ when the valuation is clear.

We now proceed with a more precise description of $\phi^0(\CC_v)$, depending somewhat on the flavour of the valuation $v$. 
\begin{lemma}
The set $\phi^0_\mathrm{Berk}$ is a (possibly degenerate) disk in $\Aberk{\CC_v}$.
\end{lemma}

\begin{proof}
Let $\preceq$ denote the usual partial ordering on $\Aberk{\CC_v}$, so that if $\zeta_1, \zeta_2\in\Aberk{\CC_v}$, then $\zeta_1\preceq \zeta_2$ if and only if $\|f\|_{\zeta_1}\leq \|f\|_{\zeta_2}$ for all $f\in \CC_v[x]$.  We note that $\fj{\phi}(\CC_v)$ is downward-closed with respect to $\preceq$, and hence is a union of (Berkovich) disks, noting that the Berkovich disk $\mathcal{D}(a, r)$ is defined to be the set of $\zeta\in\Aberk{\CC_v}$ such that $\zeta\preceq \zeta_{a, r}$ in the notation of \cite{baker-rumely}; the point $a\in \CC_v$ is simply the disk $\mathcal{D}(a; 0)$.  It follows that $\phi^0_\mathrm{Berk}$ is the largest Berkovich disk containing $0$, and contained in $\fj{\phi}_\mathrm{Berk}$.
\end{proof}

Before describing the structure of $\phi^0(\CC_v)$, we introduce one new piece of notation.  Fix a basis $T_1, ..., T_m$ for $A$ as an $\FF_q$-algebra, as above.  For a Drinfeld module $\phi/\CC_v$, let $a_{i, j}\in\CC_v$ be defined by
\[\phi_{T_i}(x)=T_ix+a_{i, 1}x^q+\cdots+a_{i, r\deg(T_i)}x^{q^{t\deg(T_i)}},\]
for each $i$, and set
\[j_{\phi, v}=\max_{\substack{1\leq i\leq m\\1\leq j\leq r\deg(T_i)}}\left\{\frac{1}{q^j-1}\log|a_{i, j}|\right\}+c(\phi).\]
In the context of a global function field, $j_{\phi, v}$ will be the local contribution to the height of $j_\phi$.
Similarly, for each non-constant $T\in A$, we let $j_{\phi_T, v}=j_{\phi\mid \FF_q[T], v}$.  In other words, we set
\[j_{\phi_T, v}=\max_{1\leq j\leq r\deg(T)}\left\{\frac{1}{q^j-1}\log|a_{j}|\right\}+c(\phi).\]

\begin{lemma}
We have $j_{\phi, v}\geq 0$, and if $\phi/\CC_v$ has potentially good reduction, then $j_{\phi, v}=0$.  If $v$ is a finite place, i.e.\ if $A$ is a subring of the ring of integers, then $j_{\phi, v}=0$ if and only if $\phi/L$ has potentially good reduction.
\end{lemma}

\begin{proof}
The first claim is trivial, since $c(\phi)$ is equal to at least one of the terms appearing in the maximum defining $j_{\phi, v}$.  If $\phi$ has potentially good reduction then, since $j_{\phi, v}$ is an isomorphism invariant, we may assume without loss of generality that $\phi$ has good reduction.  Then the coefficients of $\phi_a(x)$ are integral for all $a\in A$, and for some generator $T_i$ of $A$ over $\FF_q$, one of the coefficients of $\phi_{T_i}(x)$ is a unit.  This gives $j_{\phi, v}=0$.

On the other hand, suppose that $j_{\phi, v}=0$.  Then we may choose an $\alpha\in L_v^\mathrm{sep}$ such that $c(\alpha^{-1}\phi\alpha)=0$.  Now,
$j_{\alpha^{-1}\phi\alpha, v}=j_{\phi, v}=0$, and considering the definition of $j_{\alpha^{-1}\phi\alpha, v}$ we see that the coefficients of $(\alpha^{-1}\phi\alpha)_{T_i}(x)$ are integral, for all $i$, and at least one is a unit.  In particular, $\phi$ is $L^\mathrm{sep}$-isomorphic to a Drinfeld module of good reduction.
\end{proof}

\begin{remark}
Defining $j_{E, v}=\log^+|j_E|_v$, for an elliptic curve $E$ over a local field with valuation $v$, we also have $j_{E, v}\geq 0$ with equality if and only if $E$ has potentially good reduction.
\end{remark}

The following lemma describes $\phi^0(\CC_v)$ in the case where $v$ is a finite place.  A useful corollary of this lemma, however, is that just like $c(\phi)$, the quantity $j_{\phi_T, v}$ does not depend on the choice of $T\in A$.  In particular, we have $j_{\phi,v}=j_{\phi_T, v}$ for any non-constant $T\in A$, which allows us to proceed in some of our proofs as if we were dealing only with $\FF_q[T]$-modules.
\begin{lemma}\label{lem:finitephi0}
Suppose that $v$ is a finite place, and fix a non-constant $T\in A$.  Then
\[\phi^0_\mathrm{Berk}=\mathcal{D}(0, q^{-j_{\phi_T, v}+c(\phi)}),\]
 and $\phi^0(\CC_v)$ is a submodule of $\fj{\phi}(\CC_v)$.
\end{lemma}

\begin{proof}
Note that if $\alpha\psi=\phi\alpha$, with $\alpha\in\CC_v^*$, then $c(\psi)=c(\phi)+\log|\alpha|$ and $j_{\psi_T, v}=j_{\phi_T, v}$, and $\fj{\psi}_\mathrm{Berk}$ is the image of $\fj{\phi}_\mathrm{Berk}$ under a scaling by $\alpha^{-1}$.  In particular, the statement of the lemma is true for $\phi$ if and only if it is true for $\psi$ and so we may assume, without loss of generality, that $j_{\phi_T, v}=c(\phi)$.  In other words, we assume that every coefficient of $\phi_{T}$ is integral, and at least one is a unit.  Our aim, in this case, is to show that $\phi^0_\mathrm{Berk}=\mathcal{D}(0, 1)$.

As noted, we have $\phi^0_\mathrm{Berk}=\mathcal{D}(0, c)$, for some $c\geq 0$.  Our assumption on the coefficients of $\phi_{T}$ ensure that $\phi_{T}(\mathcal{D}(0, 1))\subseteq \mathcal{D}(0, 1)$, and so $\mathcal{D}(0, 1)\subseteq \fj{\phi}_\mathrm{Berk}$.  In other words, $c\geq 1$.   On the other hand, there is certainly some $\zeta\in\AA^1(\CC_v)\setminus \fj{\phi}(\CC_v)$, necessarily with $|\zeta|_v>1$.  We will show that for any $r>1$ there is some $x\in \mathcal{D}(0, r)$ such that $\zeta\in \Ocal_A(x)$.  It follows that $x\not\in \fj{\phi}_\mathrm{Berk}$, and consequently that $\mathcal{D}(0, r)\not\subseteq \fj{\phi}_\mathrm{Berk}$, for $r>1$.

Write
\[\phi_T(x)=Tx+a_1x+\cdots +a_{r\deg(T)}x^{q^r}.\]
Now, suppose that for some non-constant $B\in A$ we have $\phi_B(x)=Bx+\sum b_ix^{q^i}$, where each $b_i$ is integral, and at least one is a unit.  Suppose that $I$ is the largest index with $a_I$ a unit, and $J$ is the largest index with $b_j$ a unit.
Then
\[\phi_{BT}(x)=\phi_{T}(\phi_B(x))=\sum a_ib_j^{q^i} x^{q^{i+j}}.\]
The coefficient of $x^{q^{I+J}}$ is
\[\sum_{i+j=I+J}a_ib_j^{q^i}=a_Ib_J+\sum_{\substack{i>I}{i+j=I+J}}a_ib_j^{q^i}+\sum_{\substack{j>J}{i+j=I+J}}a_ib_j^{q^i},\]
which is a unit, since $|a_i|<1$ for $i>I$ and $|b_j|<1$ for $j>J$.  If $k>I+J$, then the coefficient of $x^{q^j}$ in $\phi_{BT}(x)$ is
$\sum_{i+j=k}a_ib_j^{q^i}$, where in every term either $i>I$ or $j>J$.  So all subsequent coefficients are non-units.  By induction, we see that the coefficient of $x^{q^{NI}}$ in $\phi_T^N(x)=\phi_{T^N}(x)$ is a unit (and, further, that this is the largest index corresponding to a unit coefficient in $\phi_T^N$).  In particular, the Newton polygon of the polynomial $\phi_T^N(x)-\zeta$ contains a line segment of slope at most $-v(\zeta)/(q^{NI}-1)$, and hence $\phi_T^N(x)-\zeta$ has a root $\beta\in\CC_v$ satisfyingt $|\beta|=|\zeta|^{1/(q^{NI}-1)}$.  Choosing $N$ large enough, we can ensure that $|\beta|\leq r$, and so $\mathcal{D}(0, r)\not\subseteq \fj{\phi}_\mathrm{Berk}$.  Since $r>1$ was arbitrary, we have shown that $\phi^0_\mathrm{Berk}=\mathcal{D}(0, 1)$.

It remains to show that $\phi^0(\CC_v)$ is a submodule of $\phi(\CC_v)$.  It is clear that $\phi^0(\CC_v)=D(0, 1)$ is an additive subgroup of $\CC_v$, and the observation $\phi_T(D(0, 1))\subseteq D(0, 1)$ made above suffices to show that it is also closed under the action of $A$.
\end{proof}

At this point we pause to note that, if $\Ocal_v$ is the ring of integers of $\CC_v$, with maximal ideal $\mathfrak{m}$, and $\phi$ is a Drinfeld module defined over $\Ocal_v$, with stable reduction (assuming, still, that $|T|_v\leq 1$), then we have $j_{v, \phi}=c(\phi)$.  In this case, $\phi^0(\CC_v)=\Ocal_v$, and the induced Drinfeld module $\tilde{\phi}$ over the residue field $\Ocal_v/\mathfrak{m}$ is precisely the image of $\phi^0(\CC_v)$, that is, we have $\tilde{\phi}(\Ocal/\mathfrak{m})\cong \phi^0(\CC_v)/\phi^1(\CC_v)$, where $\phi^1(\CC_v)=\{x\in \CC_v:|x|_v<1\}$, a the maximal proper submodule of $\phi^0(\CC_v)$.  It is also the case that the structure of $\phi^1(\CC_v)$ is closely related to that of a certain formal Drinfeld module \cite{rosen}, much as in the elliptic case.

This is closely analogous to the case of an elliptic curve $E/\QQ_p$, where \[E^0(\QQ_p)/E^1(\QQ_p)\cong \tilde{E}(\FF_p),\]
and it is in part this relation which motivates the study of $\phi^0(\CC_v)$.

\begin{lemma}
If $v$ is an infinite place, then $\phi^0_\mathrm{Berk}=\{0\}$.
\end{lemma}

\begin{proof}
As before, we assume without loss of generality that $j_{\phi, v}=c(\phi)$, and let $\phi^0_\mathrm{Berk}=\mathcal{D}(0, ñr)$ for some $r\geq 0$.  Note that, for any fixed $T\in A\setminus\FF_q$, the coefficients of all but the linear term of $\phi_T(x)$ are integral.  If we select any $\zeta\not\in\fj{\phi}(\CC_v)$, we see that the initial line segment in the Newton polygon of $\phi_T^N(x)-\zeta$ has slope $(Nv(T)-v(\zeta))/(q-1)$, as soon as $N$ is large enough, since $v(\zeta)$ is fixed, and $v(T)<0$.  It follows that $\phi_T^N(x)-\zeta$, again if $N$ is large enough, has a root $\beta\in\CC_v$ satisfying
\[|\beta|=|\zeta/T^N|^{1/(q-1)}.\]
Note that, as in the previous lemma, $\beta\not\in \fj{\phi}_\mathrm{Berk}$, and so $r< |\zeta/T^N|^{1/(q-1)}$.  But $|\zeta|$ is fixed, and $|T|>1$, so we obtain a contradiction for large $N$ if $r>0$.  It must be the case that $r=0$.
\end{proof}

We noted that Lemma~\ref{lem:finitephi0} gave us the convenient corollary that $j_{\phi, v}=j_{\phi_T, v}$ for any non-constant $T\in A$, in the case where $v$ is a finite place.  This is not true for infinite places, as we can see from the Carlitz module.  If $\phi_T(x)=Tx+x^q$, and $v$ is a place with $|T|>1$, then $j_{\phi_T, v}=0$, while $j_{\phi_{T^2}, v}=\frac{q}{q-1}\log|T|>0$.  In fact, it is not hard to show that in this case $j_{\phi_T^N, v}\to\infty$ as $N\to\infty$.

\begin{remark}
We defined $\fj{\phi}(\CC_v)$ to be the set of $x\in\phi(\CC_v)$ with bounded $A$-orbit under $\phi$, and similarly for subfields.  It is perhaps more natural, however, to describe $\fj{\phi}(\CC_v)$ as simply the maximal bounded submodule of $\phi(\CC_v)$, as we did in the introduction.  It is clear that $\fj{\phi}(\CC_v)$ is a bounded submodule, but if $x\in \phi(\CC_v)\setminus\fj{\phi}(\CC_v)$, then \emph{no} submodule containing $x$ is bounded, since the submodule generated by $x$ is not bounded.  So every bounded submodule of $\phi(\CC_v)$ is contained in $\fj{\phi}(\CC_v)$.
\end{remark}

%%%%%%%%%%%%%%%%%%%%%%%%%%%%%%%%%%%
%%%%
%%%%%%%%%%%%%%%%%%%%%%%%%%%%%%%%%%%
%%%%
%%%%%%%%%%%%%%%%%%%%%%%%%%%%%%%%%%%
%%%%
%%%%%%%%%%%%%%%%%%%%%%%%%%%%%%%%%%%
%%%%
%%%%%%%%%%%%%%%%%%%%%%%%%%%%%%%%%%%
%%%%
%%%%%%%%%%%%%%%%%%%%%%%%%%%%%%%%%%%
%%%%
%%%%%%%%%%%%%%%%%%%%%%%%%%%%%%%%%%%
%%%%
%%%%%%%%%%%%%%%%%%%%%%%%%%%%%%%%%%%

\section{Local heights for the filled Julia set}
\label{sec:localheights}

Here we set down some basic results about local heights on the filled Julia set.  Some of the results in this section are special cases of results in \cite{jpaper}.

\begin{definition}
Let  $\phi/\CC_v$ be a Drinfeld module with filled Julia set $\fj{\phi}(\CC_v)$.  Define the \emph{local height} associated to $\fj{\phi}$ by
\[\lambda_{\phi}(x)=\log|x^{-1}|+c_v(\phi).\]
\end{definition}

\begin{remark}
Note that this agrees with the definition of the local height associated to a Drinfeld module in \cite{jpaper}, since the Greens' function vanishes identically on $\fj{\phi}(\CC_v)$.
\end{remark}

\begin{lemma}\label{lem:lambdaprops}
Let $\lambda_\phi$ be the local height associated to $\fj{\phi}/\CC_v$.
\begin{enumerate}
\item The map $\lambda_\phi:\fj{\phi}(\CC_v)\to \RR$ is continuous, except for a simple pole at the origin.
\item\label{item:invariance} If $\psi\alpha=\alpha\phi$, for some $\alpha\in\CC_v^*$, then \[\lambda_\phi(x)=\lambda_\psi(\alpha x).\]
\item If $\phi$ has good reduction, then
\[\lambda_\phi(x)=\log|x^{-1}|.\]
\item If $v$ is a finite place, then
\[\lambda_\phi(x)\geq -\left(\frac{1-q^{-(r-1)}}{q-1}\right)j_{\phi, v},\]
for all $x\in \fj{\phi}(\CC_v)$.
\item  If $v$ is a finite place, then
\[\lambda_\phi(x)\geq j_{\phi, v},\]
For all $x\in \phi^0(\CC_v)$.
\item The function $\lambda_{\phi}$ is constant on every non-trivial coset $x+\phi^0(\CC_v)$.  In particular, there is a function $\mathbb{B}:\fj{\phi}(\CC_v)/\phi^0(\CC_v)\to \RR$ and a non-negative function $E:\fj{\phi}(\CC_v)\to \RR$ such that $E$ vanishes except on $\phi^0(\CC_v)$ and
\[\lambda_{\phi}(x)=E(x)+\mathbb{B}(x+\phi^0(\CC_v))\]
for all $x\in \fj{\phi}(\CC_v)$.
\end{enumerate}
\end{lemma}

\begin{remark}
Compare with Theorem~4.2 of \cite[p.~473]{ataec} for local heights relative to elliptic curves over $p$-adic fields.  In particular, the role of the second Bernoulli polynomial there is played by our function $\mathbb{B}$ here, which will be given a more explicit description below.
\end{remark}

\begin{proof}[Proof of Lemma~\ref{lem:lambdaprops}]
\begin{enumerate}
\item This result is clear enough, but also follows from \cite[Proposition~8.66 p.~242]{baker-rumely}.
\item If $T\in A$ is non-constant, and \[\psi_T(x)=\alpha\phi_T(\alpha^{-1}z),\] we have
 $c_v(\psi)=c_v(\phi)-\log|\alpha|_v$, so
\begin{eqnarray*}
\lambda_{\phi}(x)&=&\log|x^{-1}|+c_v(\phi)\\
&=&\log|x^{-1}|+c_v(\psi)-\log|\alpha|\\
&=&\lambda_{\psi}(\alpha z).
\end{eqnarray*}
So the local height is coordinate invariant.

\item Note that if $\phi/\CC_v$ has good reduction, then the coefficients of $\phi_T$ are integral, and the leading coefficient is a unit.  It follows that $c_v(\phi)=0$. \item By~(\ref{item:invariance}), and the invariance of $j_{\phi, v}$ under isomorphism, it suffices to prove the estimate in the case $j_{\phi, v}=c(\phi)$, that is, the case where the coefficients of $\phi_T(x)$ are integral, and one of them is a unit.  Now, we know that $|x|\leq \max\{|\xi|:\xi\in\phi[T]\}$, and examining the Newton polygon we obtain a bound on the ....
\item If $x\in \phi^0(\CC_v)$ then, by Lemma~\ref{lem:finitephi0} we have
\[\lambda_\phi(x)=-\log|x|+c(\phi)\geq -\left(-j_{\phi, v}+c(\phi)\right)+c(\phi).\]

\item  The statement is trivial for $v$ an infinite place, since $\phi^0(\CC_v)=\{0\}$; we may take $\mathbb{B}(x+\phi^0(\CC_v))=\lambda_\phi(x)$ for all $x$.  So suppose that $v$ is a finite place. We have $G_\phi$ vanishing identically on $\fj{\phi}(\CC_v)$, and so $\lambda_\phi(x)=-\log|x|+c(\phi)$ here.  In particular, $\lambda_{\phi}(x)$ is a function of $|x|_v$.  But $|x|_v$ is constant on every non-trivial component of $\fj{\phi}(\CC_v)$, since every such component has the form $\zeta+D(0, q^{-J_v(\phi)})$ for some $\zeta\not\in D(0, q^{-J_v(\phi)})$.  So if we set
\[\mathbb{B}(x+\phi^0(\CC_v))=\min_{y\in x+\phi^0(\CC_v)}\lambda_{\phi}(y),\]
we see that $\lambda_{\phi}(x)=\mathbb{B}(x+\phi^0(\CC_v))$ for every $x\in \fj{\phi}(\CC_v)\setminus\phi^0(\CC_v)$.  It is automatic that $E(x)=\lambda_\phi(x)-\mathbb{B}(x+\phi^0(\CC_v))$ is non-negative, and constant on every non-trivial component of $\fj{\phi}(\CC_v)$.  
\end{enumerate}
\end{proof}

The following definition and lemmas are motivated by the arguments of Ghioca~\cite{ghioca1, ghioca2}.
\begin{definition}
If $\phi/L$ is a Drinfeld module, and $T\in A$ is non-constant, then we say that $x\in L$ is \emph{$T$-generic} if
\[|\phi_T(x)|=\max_{0\leq i\leq r\deg_\infty(T)}|a_ix^{q^i}|,\]
where $a_0=T$ as usual.  
\end{definition}

The next lemma, which says essentially that the typical element of $\fj{\phi}(\CC_v)$ is at most the size of the average torsion, is a modification of a result in \cite{jpaper}.

\begin{lemma}\label{lem:genericbound}
If $x\in \fj{\phi}(\CC_v)$ is $T$-generic
then
\[\log|x|\leq c(\phi)+\frac{1}{(q^{r\deg(T)}-1)^2}\log^+|T^{-1}|.\]
\end{lemma}

\begin{proof}
This follows directly from Lemma~4.6 of \cite{jpaper} and Lemma~\ref{lem:juliadisk} above.
\end{proof}

\begin{lemma}
If $x, \phi_T(x)\in\fj{\phi}(\CC_v)$ are $T$-generic, then
\[\lambda_\phi(x)\geq (1-q^{-1})j_{\phi_T, v}-\frac{1}{q(q^{r\deg(T)-1})^2}\log^+|T^{-1}|.\]
\end{lemma}

\begin{proof}
This follows from Lemma~4.8 in \cite{jpaper}.
\end{proof}

\begin{lemma}\label{lem:localchoose}
Suppose that $X\subseteq \fj{\phi}(L)$ is an additive subgroup, and that $T\in A$ is non-constant.
Then there is an additive subgroup $Y\subseteq X$ with $\#Y\geq q^{-4r^2\deg(T)^2} \# X$ such that for every non-zero $x\in Y$, we have
\[\lambda_{\phi}(x)\geq  (1-q^{-1})j_{\phi_T, v}-\frac{1}{q(q^{r\deg(T)-1})^2}\log^+|T^{-1}|.\]
\end{lemma}

\begin{proof}
To begin, we note that it suffices to construct a sub\emph{set} $Y\subseteq X$ satisfying the criteria, since the inequality
\[\lambda_{\phi}(x+y)\geq\min\{\lambda_{\phi}(x), \lambda_{\phi}(y)\}\]
will ensure that the subgroup generated by $Y$ will also satisfy the criteria.  Second of all, by the previous lemma, we note that it suffices to find a subset $Y\subseteq X$ of the appropriate size, such that for all $x\in Y$, both $x$ and $\phi_{T}(x)$ are $T$-generic.  As usual, we will set $R=r\deg(T)$ for notational simplicity.

For any set $W\subseteq \fj{\phi}(L)$, we define $N_1(W)$ so that $q^{N_1(W)}$ is the number of elements $\xi\in\phi[T]$ such that $|\xi|\leq \max\{|x|:x\in W\}$.
Note that, by examining the Newton polygon of $\phi_T(x)$, we see that $N_1(W)$ is always an integer, and $0\leq N_1(W)\leq R$ as long as $W$ is non-empty.
Similarly, define
$N_T(W)=N_1(\phi_T(W))$, that is, $q^{N_T(W)}$ is the number of $\xi\in\phi[T]$ such that $|\xi|\leq \max\{|\phi_{T}(x)|:x\in W\}$.

For $\xi\in\phi[T]$, define
\[Z_\xi=\{y\in\phi(\CC_v):|y-\xi|<|\xi|=|y|\}\]
if $\xi\neq 0$, and
\[Z_0=\{y\in\phi(\CC_v):y\text{ is generic}\}.\]
Note that $\phi(\CC_v)=\bigcup_{\xi\in\phi[T]}Z_\xi$.
Let $X_1=X$, and suppose that for $i\geq 1$ we have constructed a set $X_i$.  For each $(\xi_1, \xi_2)\in \phi[T]\times\phi[T]$, let
\[X_{i, \xi_1, \xi_2}=\left\{x\in X_i:x\in Z_{\xi_1}\text{ and }\phi_{T}(x)\in Z_{\xi_2}\right\}.\]
Clearly the sets $X_{i, \xi_1, \xi_2}$ cover $X_i$, and there are $(q^R)^2$ of them, so there is some pair $(\xi_1, \xi_2)$ such that $\# X_{i, \xi_1, \xi_2}\geq q^{-2R}\# X_i$.

We consider several cases.

\noindent\textbf{Case 1:} $\xi_1\neq 0$.

Choose some $x\in X_{i, \xi_1, \xi_2}$, and let $X_{i+1}=X_{i, \xi_1, \xi_2}-x$, so that $\# X_{i+1}=\# X_{i, \xi_1, \xi_2}$.  Since
\[\max\{|\phi_{T}(x)-\phi_{T}(y)|:y\in X_{i, \xi_1, \xi_2}\}\leq\max\{|\phi_{T}(y)|:y\in X_i\},\]
we see that $N_T(X_{i+1})\leq N_T( X_{i})$.  On the other hand,
\[|x-y|=|y-\xi_1+x-\xi_1|<|\xi_1|=|y|=|x|\]
for all $x, y\in X_{i, \xi_1, \xi_2}$, by definition.  It follows that
\[N_1(X_{i+1})<N_1(X_{i, \xi_1, \xi_2})\leq N_1(X_i).\]
In particular, since $N_1(W)$ is a non-negative integer for any non-empty $W\subseteq \fj{\phi}(L)$, the present case can arise for at most $R$ values of $i$.

\noindent\textbf{Case 2:} $\xi_1=0$, $\xi_2\neq 0$.

In this case we again choose some $x\in X_{i, \xi_1, \xi_2}$, and let $X_{i+1}=X_{i, \xi_1, \xi_2}-x$.  Arguments just as in Case 1 show that here $N_1(X_{i+1})\leq N_1(X_i)$, while $N_T(x_{i+1})<N_T(X_i)$.  In particular, this case also arise for at most $R$ values of $i$.

\noindent\textbf{Case 3:} $\xi_1=\xi_2=0$.

In this case, we may take $Y=X_{i, 0, 0}$, since then $x$ and $\phi_{T}(x)$ are both $T$-generic for all $x\in Y$, by construction.  We then have that $\# Y\geq q^{-2Ri}\# X$, and we have seen that we arrive in this case with $i\leq 2R$.  This proves the lemma.
\end{proof}

Note that on each non-trivial class of $\fj{\phi}(L)/\phi^0(L)$, the absolute value function is constant, and so we may speak unambiguously about $|x+\phi^0(L)|$, as long as $x\not\in \phi^0(L)$.  If $T\in A$ is non-constant, we will say that a class $x\in \fj{\phi}(L)/\phi^0(L)$ is $T$-\emph{generic} if and only if
\[|\phi_T(x)|=\max_{0\leq i\leq r\deg(T)} |a_ix|^{q^i}.\]
In other words, in the language of Berkovich spaces, if $\zeta\in\Aberk{\CC_v}$ is the point corresponding to $x$, then $x$ will be $T$-generic if and only if
\[0 \vee \phi_T(\zeta)=\phi_T(0\vee\zeta),\]
where $x\vee y$ is the least upper bound of $x, y\in \Aberk{\CC_v}$ \cite[\S~1.4]{baker-rumely}.
In agreement with this, the trivial class will always be deemed generic.

\begin{lemma}\label{lem:vjbound}
There is a bound on $\#(\fj{\phi}(L_v)/\phi^0(L_v))$ depending only on $v(j_\phi)$ and $s$, where $\phi$ has potentially stable reduction of rank $r-s$.
\end{lemma}

\begin{proof}
As above, for $\xi\in(\phi[T]+\phi^0(L_v))/\phi^0(L_v)$, let
\[Z_\xi=\{x+\phi^0(L_v):|x-\xi|<|\xi|=|x|\},\]
and let
\[Z_0=\{x+\phi^0(L_v)\text{ generic}\},\]
where we take the trivial coset to be generic.  Similarly to the above, for any set $W\subseteq \fj{\phi}(L_v)/\phi^0(L_v)$, we define $N(W)$ so that $q^{N(W)}$ is the number of elements $\xi\in(\phi[T]+\phi^0(L_v))/\phi^0(L_v)$ such that $|\xi|\leq \max\{|x|:x\in W\}$.  Note that there are $q^s$ distinct elements of $(\phi[T]+\phi^0(L_v))/\phi^0(L_v)$, and so we have $0\leq N(W)\leq s$ for all $W\subseteq \fj{\phi}(L_v)/\phi^0(L_v)$.

Now, let $G\subseteq \fj{\phi}(L_v)/\phi^0(L_v)$ be an additive subgroup containing at least $2q^{(ns+1)ns}$ distinct elements.  For each $i\geq 0$ we will construct a subset $X_i\subseteq G$ such that $\# X_i\geq 2q^{(ns+1-i)ns}$.  We let $X_0=G$, and for each $i\geq 0$ and each
\[(\xi_1, ..., \xi_n)\in \left((\phi[T]+\phi^0(L_v))/\phi^0(L_v)\right)^n\]
we let
\[X_{i, \xi_1, \xi_2, ..., \xi_n}=\{x\in X_i:x\in Z_{\xi_1}, \phi_T(x)\in Z_{\xi_2}, ..., \phi_{T^{n-1}}(x)\in Z_{\xi_n}\}.\]
Note that there are $q^{sn}$ sets $X_{i, \xi_1, ..., \xi_n}$, and that these cover $X_i$.  In particular, since $X_i$ contains at least $q^{(ns-i)(ns)}$ elements, there is at least one subset of the form $X_{i, \xi_1, ..., \xi_n}$ containing at least $q^{(ns-i-1)ns}$ elements.

First we consider the case where $\xi_j\neq 0$, for some $1\leq j\leq n$.  In this case we choose some $x\in X_{i, \xi_1, ..., \xi_n}$, and let $X_{i+1}=x-X_{i, \xi_1, ..., \xi_n}$.  Note that for each $0\leq k\leq n-1$, we have
\[N(\phi_{T^k}(X_{i+1}))\leq N(\phi_{T^k}(X_{i, \xi_1, ..., \xi_n}))\leq N(\phi_{T^k}(X_i)).\]
On the other hand, for all $y\in X_{i, \xi_1, ..., \xi_n}$, we have by definition
\[|\phi_{T^{j-1}}(x-y)|=|\phi_{T^{j-1}}(x)-\xi_j-\phi_{T^{j-1}}(y)+\xi_j|<|\xi_j|=|\phi_{T^{j-1}}(x)|=|\phi_{T^{j-1}}(y)|,\]
and so
\[N(\phi_{T^{j-1}}(X_{i+1}))<N(\phi_{T^{j-1}}(X_{i, \xi_1, ..., \xi_n}))\leq N(\phi_{T^{j-1}}(X_i)).\]
In other words, if we set
\[M_i=\sum_{k=0}^{m-1}N(\phi_{T^k}(X_i)),\]
for each $i$, we have $M_{i+1}<M_i$, and so $0\leq M_i\leq ns-i$.  In particular, the case where $\xi_j\neq 0$ for some $j$ occurs at most $ns$ times.  

So for some $i\leq ns$, we have $\# X_{i, 0, 0, ..., 0}\geq 2q^{(ns+1-i-1)ns}\geq 2$.  In particular, there is some non-trivial $x\in \fj{\phi}(L_v)/\phi^0(L_v)$ such that $x, \phi_T(x), \phi_{T^2}(x), ..., \phi_{T^{n-1}}(x)$ are all $\phi_T$-generic.   Note that if $y\in\fj{\phi}(L_v)$ is $\phi_T$-generic and $|\phi_T(y)|\leq |y|$, then $|a_iy^{q^i}|\leq |y|$ for all $i$.  Re-arranging this, we have that $\log|y|\leq -\frac{1}{q^i-1}\log|a_i|$ for all $i$, and hence that $y\in \phi^0(L_v)$.  In other words, if $y$ is $\phi_T$-generic and $y+\phi^0(L_v)$ is non-trivial in $\fj{\phi}(L_v)/\phi^0(L_v)$, we must have $-v(y)<-v(\phi_T(y))$.  Applying this to our special point $x$, we have that $-v(\phi_{T^{j}}(x))\geq -v(x)+j$ for all $0\leq j<n$.  Since $x+\phi^0(L_v)$ is non-trivial, we have $-j_{\phi, v}+c(\phi)<\log|x|=-v(x)\deg(v)$.  But, since $\phi_{T^{n-1}}(x)$ is $\phi_T$-generic, we also have
\[-v(\phi_{T^{n-1}}(x))\deg(v)=\log|\phi_{T^{n-1}}(x)|\leq c(\phi).\]
Combining these, we have
\begin{multline*}
-j_{\phi, v}+c(\phi)+(n-1)\deg(v)<-v(x)\deg(v)+(n-1)\deg(v)\\\leq -v(\phi_{T^{n-1}}(x))\deg(v)\leq c(\phi),
\end{multline*}
and hence
\[n-1< \frac{j_{\phi, v}}{\deg(v)}=v(j_\phi).\]
This is a contradiction if we choose an integer $n$ with  $v(j_\phi)+1\leq n< v(j_\phi)+2$, and hence every additive subgroup of $\fj{\phi}(L_v)/\phi^0(L_v)$ contains fewer than \[2q^{(v(j_\phi)s+2s+1)(v(j_\phi)+2)s}\] distinct elements.  In particular, $\fj{\phi}(L_v)/\phi^0(L_v)$ itself contains fewer than this number of elements.
\end{proof}

\begin{remark}
An examination of the proof above shows that $i$ never gets larger than $n(s'-1)$, where $s'$ is the number of distinct slopes in the Newton polygon of $\phi_T$, or equivalently, the number of distinct sizes of elements of $\phi[T]$.  Since the bound in the lemma still ends up depending on $s$, however, this improvement does not seem particularly useful.
\end{remark}

%%%%%%%%%%%%%%%%%%%%%%%%%%%%%%%%%%%%%%%%%%%%%%%%%
%%%%%%%%%%%%%%%%%%%%%%%%%%%%%%%%%%%%%%%%%%%%%%%%%
%%%%%%%%%%%%%%%%%%%%%%%%%%%%%%%%%%%%%%%%%%%%%%%%%
%%%%%%%%%%%%%%%%%%%%%%%%%%%%%%%%%%%%%%%%%%%%%%%%%
%%%%%%%%%%%%%%%%%%%%%%%%%%%%%%%%%%%%%%%%%%%%%%%%%
%%%%%%%%%%%%%%%%%%%%%%%%%%%%%%%%%%%%%%%%%%%%%%%%%

\section{Drinfeld modules over global fields}
\label{sec:global}

In this section we let $A$ and $K$ be as in Section~\ref{sec:prelims}, and take $L/K$ to be a finite extension.  

\begin{definition}\label{def:mudef}
For any Drinfeld $A$-module $\phi/L$, we set
\[S_\phi(\mathfrak{a})=\left\{v\in M_L^0:\mathfrak{a}\left(\fj{\phi}(L_v)/\phi^0(L_v)\right)\neq \{0\}\right\}.\]
If, in addition, $N\geq 0$, then we set
\[\mu(\phi, N, \mathfrak{a})=\max_{\substack{S\subseteq M_L^0\\ \# S\leq N}}\left\{\frac{\sum_{v\in M_L^0\setminus (S_\phi(\mathfrak{a})\cup S)}j_{\phi, v}}{\sum_{v\in M_L^0\setminus S}j_{\phi, v}}\right\},\]
where sums over empty sets take the value 1, and we take $\mu(\phi, N, \mathfrak{a})=1$ by convention if the denominator vanishes.
\end{definition}

Before proceeding with the proof of Theorem~\ref{th:szpiroresult}, we recast the statement of Conjecture~\ref{question} in the case $N=0$.  For any finite set $S\subseteq M_L$, let $\mathbf{A}_{L, S}$ denote the $S$-free adeles of $L$, that is, the restricted product
\[\mathbf{A}_{L, S}={\prod_{v\in M_L\setminus S}}^*L_v,\]
consisting of elements of the product whose entries are integral except at finitely many places.  Note that the module structure $\phi(L)$ extends naturally to a module structure $\phi(\mathbf{A}_{L, S})$, and we write $\fj{\phi}(\mathbf{A}_{L S})$ for those elements $\alpha\in \phi(\mathbf{A}_{L, S})$ such that $\alpha_v\in \fj{\phi}(L_v)$ for all $v$, and similarly for $\phi^0(\mathbf{A}_{L, S})$.
\begin{theorem}
Conjecture~\ref{question}, with $N=0$, is equivalent to the following statement: There exists an ideal $\mathfrak{a}\subseteq A$ such that for every Drinfeld $A$-module $\phi/L$ of rank $r$, there exists a finite set $S\supseteq M_L^\infty$ of places such that $\fj{\phi}(\mathbf{A}_{L, S})/\phi^0(\mathbf{A}_{L, S})$ is $\mathfrak{a}$-torsion, and
\[ \sum_{v\in S}j_{\phi, v}\leq \frac{1}{q}\sum_{v\in M_L^\infty}j_{\phi, v}+\left(1-\frac{1}{q}\right)h(j_\phi).\]
\end{theorem}

We begin with a simple lemma which will allow us to consider only the case in which $h(j_\phi)$ is large.
\begin{lemma}\label{lem:largej}
For any $B\geq 0$, there is a uniform bound on $\#\phi^{\mathrm{Tors}}(L)$ as $\phi/L$ varies over Drinfeld modules of rank $r$ with $h(j_\phi)\leq B$.
\end{lemma}

\begin{proof}
As noted above, the height $h$ on weighted projective space satisfies the Northcott property, and so in particular there are only finitely many $j\in M_{A, r}(L)$ with $h(j)\leq B$.  It suffices, then, to prove the claim for any single value of $j\in M_{A, r}(L)$.  But if $\phi/L$ and $\psi/L$ satisfy $j_\phi=j_\psi$, then there is an $\alpha\in L^\mathrm{sep}$ such that $\alpha\psi=\phi\alpha$.  Furthermore, $\alpha^{q^k-1}\in L$ for some $k\leq r\max_{1\leq i\leq m}\{\deg(T_i)\}$, if $T_1, ..., T_m$ generate $A$ as an $\FF_q$-algebra.  Theorem~3 of \cite{poonenunif} shows that $\#\psi^{\mathrm{Tors}}(L)$ is bounded uniformly for twists of a fixed $\phi/L$ over extensions of bounded degree.
\end{proof}

\begin{proof}[Proof of Theorem~\ref{th:szpiroresult}]
Invoking Lemma~\ref{lem:largej}, we will assume that
\[h(j_\phi)>  \frac{q^{r-1}}{(q^r-1)^2}\sum_{1\leq i\leq m}\deg(T_i).\]

Suppose that, for some given $\mathfrak{a}\subseteq A$ and $N$ we have $\mu(\phi, N, \mathfrak{a})\geq 1/q$, so that for some set $S\subseteq M_L^0$ containing at most $N$ places, we have
\[\sum_{v\in M_L^0\setminus  (S_\phi(\mathfrak{a})\cup S)}j_{\phi, v}\geq \frac{1}{q}\sum_{v\in M_L^0\setminus S}j_{\phi, v}.\]

Let $T_1, ..., T_m$ be a generating set for $A$ as an $\FF_q$-algebra.  Applying Lemma~\ref{lem:localchoose} to $X=\mathfrak{a}\phi^{\mathrm{Tors}}(L)\subseteq \phi^{\mathrm{Tors}}(L)$, we see that there is a subgroup $Y$ with \[\#Y\geq \left(q^{-4r^2\sum_{i=1}^m\deg(T_i)[L:K]}\right) \#X\] such that for every $v\in M_L^\infty$, every $1\leq i\leq m$, and every $x\in Y$, we have either $x=0$, or else
\[\lambda_{\phi, v}(x)\geq (1-q^{-1})j_{\phi_{T_i}, v}-\frac{1}{q(q^{r\deg(T_i)-1})^2}\log^+|T_i^{-1}|.\]
Note that, for any place $v\in M_L$, the quantity
$j_{\phi, v}$ is simply the largest of the values $j_{\phi_{T_i, v}}$,
and so we have
\[\lambda_{\phi, v}(x)\geq (1-q^{-1})j_{\phi, v}-\delta_v\]
for every non-zero $x\in Y$ and every place $v\in M_L^\infty$, where
\[\delta_v=\max_{1\leq i\leq m}\frac{1}{q(q^{r\deg(T_i)-1})^2}\log^+|T_i^{-1}|.\]

Similarly, we can apply Lemma~\ref{lem:localchoose} to the places in $S$ to choose a subset $Z\subseteq Y$ with
\[\# Z \geq \left(q^{-4r^2\left(\sum_{i=1}^m\deg(T_i)[L:K]+\#S\right)}\right)\# X\]
such that for all non-zero $x\in Z$ and each $v\in S$, we have
\[\lambda_{\phi, v}(x)\geq (1-q^{-1})j_{\phi, v}-\delta_v.\]

Note that for any $x\in Z$ we have $x\in \phi^0(L_v)$ for every $v\in M_L^0\setminus S_\phi(\mathfrak{a})$, since $Z\subseteq \mathfrak{a}\fj{\phi}(L)$, and $\mathfrak{a}$ annihilates the component group $\fj{\phi}(L)/\phi^0(L)$ for each $v\in M_L^0\setminus S_\phi(\mathfrak{a})$.  So in particular, for $v\in M_L^0\setminus S_\phi(\mathfrak{a})$ and $x\in Z$ non-zero, we have
\[\lambda_{\phi, v}(x)\geq j_{\phi, v}.\]  On the other hand, for every $v\in M_L^0$, and so in particular for $v\in S_\phi(\mathfrak{a})$, we have the trivial lower bound
\[\lambda_{\phi, v}(x)\geq -\left(\frac{1-q^{-(r-1)}}{q-1}\right)j_{\phi, v}.\]

So if $\mu(\phi, N, \mathfrak{a})\geq 1/q$, we have by definition
\[\frac{1}{q} \sum_{v\in M_L^0\setminus S}j_{\phi, v}\leq \sum_{v\in M_L^0\setminus  (S_\phi(\mathfrak{a})\cup S)}j_{\phi, v},\]
and
\[\left(1-\frac{1}{q}\right) \sum_{v\in M_L^0\setminus S}j_{\phi, v}\geq \sum_{v\in M_L^0\setminus (S_\phi(\mathfrak{a})\cup S)}j_{\phi, v},\]
and hence
\begin{eqnarray*}
0&=&\sum_{v\in M_L}\lambda_{\phi, v}(x)\\
&=&\sum_{v\in M_L^0\setminus  (S\cup S_\phi(\mathfrak{a}))}\lambda_{\phi, v}(x)+\sum_{v\in S_\phi(\mathfrak{a})}\lambda_{\phi, v}(x)+\sum_{v\in M_L^\infty\cup S}\lambda_{\phi, v}(x) \\
&\geq &\sum_{v\in M_L^0\setminus (S\cup S_\phi(\mathfrak{a}))}j_{\phi, v}+\sum_{v\in S_\phi(\mathfrak{a})} -\left(\frac{1-q^{-(r-1)}}{q-1}\right)j_{\phi, v}+\sum_{v\in M_L^\infty\cup S}(1-q^{-1})j_{\phi, v}-\delta_v \\
&\geq&\frac{1}{q}\sum_{v\in M_L^0\setminus S} j_{\phi, v}-\left(\frac{1-q^{-(r-1)}}{q-1}\right) \left(1-\frac{1}{q}\right)\sum_{v\in M_L^0\setminus S}j_{\phi, v}+\sum_{v\in M_L^\infty\cup S}(1-q^{-1})j_{\phi, v} -\sum_{v\in M_L}\delta_v\\
&=&\frac{1}{q^r} \sum_{v\in M_L^0\setminus S} j_{\phi, v}+\sum_{v\in M_L^\infty\cup S}(1-q^{-1})j_{\phi, v}-\sum_{v\in M_L}\delta_v\\
&\geq & \frac{1}{q^r} h(j_\phi)-\sum_{v\in M_L}\delta_v,
\end{eqnarray*}
and so
\[h(j_\phi)\leq q^r\sum_{v\in M_L}\delta_v\leq \frac{q^{r-1}}{q(q^r-1)^2}\sum_{1\leq i\leq m}\deg(T_i).\]
This contradictions our assumption, and so it must be the case $Y=\{0\}$, and hence that \[\# \mathfrak{a}\phi^{\mathrm{Tors}}(L)\leq \left(q^{-4r^2\left(\sum_{i=1}^m\deg(T_i)[L:K]+\#S\right)}\right).\]
But it is well-known that $\phi^{\mathrm{Tors}}(L)/\mathfrak{a}\phi^{\mathrm{Tors}}(L)\subseteq (A/\mathfrak{a})^r$, and so given that $\mu(\phi, N, \mathfrak{a})\geq 1/q$, we have
\[\# \phi^{\mathrm{Tors}}(L)\leq \operatorname{Norm}(\mathfrak{a})^r\left(q^{-4r^2\left(\sum_{i=1}^m\deg(T_i)[L:K]+N\right)}\right).\]
\end{proof}

We now turn to the proof of Theorem~\ref{th:simplefams}.  As mentioned in the Introduction, a more general result than Theorem~\ref{th:simplefams} may be obtained with some more technical hypotheses.
\begin{theorem}\label{th:simplefams2}
Let $L/K$ be a finite extension, let $X/\FF_q$ be a curve, let $\phi/L(X)$ be a simple family of Drinfeld modules, such that the generic fibre has at least 3 places of genuinely bad reduction, suppose that $q\neq 2, 3$, and fix $B\geq 1$.  Then the family of fibres $\phi_\beta$ with $\beta\in X(L)$ of inseparable degree at most $B$ satisfies Conjecture~\ref{question}.  In particular, $\# \phi_\beta^\mathrm{Tors}(L)$ is bounded in terms of $\deg_\mathrm{i}(\beta)$.
\end{theorem}

\begin{proof}[Proof of Theorems~\ref{th:simplefams} and \ref{th:simplefams2}]
We will define a map $J_\phi:X\to \PP^N$ by setting $J_\phi=N(j_\phi)$, where $N:\PP^{\vec{w}}\to \PP^N$ is the usual map
\[[x_0:x_2:\cdots:x_N]\mapsto \left[x_0^{w_1w_2\cdots w_N}:x_1^{w_0w_2\cdots w_N}:\cdots :x_N^{w_0w_1\cdots w_{N-1}}\right].\]
Since $L/K$ is a finite rational extension, we have $L=\FF_{q^e}(C)$, for some rational  curve $C/\FF_{q^e}$, and some integer $e\geq 1$.  For each $\beta$ we obtain an extension $L/L_\beta$ with $L_\beta=\FF_{q^e}(\beta)$, which is in turn a finite extension of $F=\FF_{q^e}(J_\phi\circ\beta)$.
We will denote by $E$ the separable closure of $F$ in $L$,  so that $L/E$ is purely inseparable, and note that $L_\beta/E_\beta$ is also purely inseparable, where $E_\beta=E\cap L_\beta$.  We have the following diagram, with degrees of extensions noted.

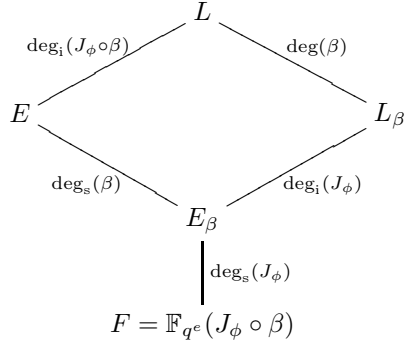
\[\xymatrix{
& L\ar@{-}[ld]_{\deg_\mathrm{i}(J_\phi\circ\beta)} \ar@{-}[rd]^{\deg(\beta)} &\\
 E\ar@{-}[rd]_{\deg_\mathrm{s}(\beta)} & & L_\beta \ar@{-}[ld]^{\deg_\mathrm{i}(J_\phi)}\\
 & E_\beta \ar@{-}[d]^{\deg_\mathrm{s}(J_\phi)} &\\
 & F=\FF_{q^e}(J_\phi\circ \beta) &
}\]

Note that the places of $L$ over which $\phi_\beta$ has persistently bad reduction are precisely the places above the place of $F$ corresponding to the point at infinity.  We will denote this set of places by $S_L$, and similarly for the intermediate fields.  Our hypothesis that the generic fibre has at least $m_q$ places of bad reduction, counted with multiplicity, corresponds to the statement
that $\sum_{v\in S_{L_\beta}}\deg(v)\geq m_q$,
where we recall that $m_q=3$ unless $q=2$ or $q=3$, in which case $m_q=5$ or $m_q=4$, respectively.

We recall, for convenience, the behaviour of places in purely inseparable extensions from \cite[Proposition~7.5, p.~81]{rosenbook}.

\begin{lemma}\label{lem:insep}
Let $F_2/F_1$ be a finite extension of function fields with perfect constant fields, and let $F^c\subseteq F_2$ be the separable closure of $F_1$ in $F_2$.  Then $g_{F^c}=g_{F_2}$, and for every place $w\in M_{F^c}$, there is a unique place $v\in M_{F_2}$ above $v$, and this place satisfies $\deg(v)=\deg(w)$ and $e_v(F_2/F^c)=[F_2:F^c]$.
\end{lemma}

In particular, we have
\[\sum_{v\in S_{E_\beta}}\deg(v)=\deg_\mathrm{s}(J_\phi)\geq m_q,\]
since every one of these places corresponds to a unique place of $L_\beta$ of the same degree.

Identifying places of $L$ with $\FF_{q^e}$-rational divisors on $C$, we note that the places of persistently bad reduction of $\phi_\beta$, that is, those with $j_{\phi_\beta, v}>0$, are precisely those appearing as summands in the pole divisor of $J_\phi\circ\beta$.  More specifically, every $v\in M_L$ appears in this divisor with coefficient $\left(\prod w_i\right) v(j_{\phi_\beta})$, where the $w_i$ are the weights defining the projective space over $L$ in which $j_{\phi_\beta}$ resides.  For brevity, we will write $W=\prod w_i$.  Note that $Wh(j_{\phi_\beta})=\deg(J_\phi\circ\beta)=[L:F]$.  We will also denote by $S_F$ the set of places corresponding to $\operatorname{Supp}(H)$, and by $S_E$ and $S_L$ the sets of places of $E$ and $L$, respectively, lying above these.

We now apply the Riemann-Hurwitz Formula \cite[p.~90]{rosenbook}, much as in the proof of \cite[Theorem~7.17]{rosenbook}, to obtain a bound on $\deg_\mathrm{s}(\beta)=[E:E_\beta]$.
Since $[E:E_\beta]$ is separable, we have
\begin{eqnarray*}
2g_E-2&\geq & -2[E:E_\beta]+\sum_{v\in M_E}\left(e_v(E/E_\beta)-1\right)\deg(v)\\
&\geq&-2[E:E_\beta]+\sum_{v\in S_E} \left(e_v(E/E_\beta)-1\right)\deg(v)\\
&=&-2[E:E_\beta]+\sum_{v\in S_E} e_v(E/E_\beta)\deg(v)-\sum_{v\in S_E}\deg(v)\\
&=&-2[E:E_\beta]+\sum_{v\in S_{E_\beta}}[E:E_\beta] \deg(v)-\sum_{v\in S_E}\deg(v)\\
&=&(m-2)[E:E_\beta]-\sum_{v\in  S_E}\deg(v),
\end{eqnarray*}
where $m=m_q=\deg_\mathrm{s}(J_\phi)$.

Now, again by Lemma~\ref{lem:insep}, every place of $E$ lifts to a unique place of $L$ of the same degree, and $g_L=g_E$.  At the same time, $[E:E_\beta]=\deg_{\mathrm{s}}(\beta)$, and so we obtain
\begin{equation}\label{eq:abc}(\deg_\mathrm{s}(J_\phi)-2)\deg_\mathrm{s}(\beta)\leq 2g_L-2+\sum_{v\in  S_L}\deg(v).\end{equation}

Now, for any $\beta\in X(L)$ and any integer $x\geq 1$, let
\[S_{x}=\{v\in M_L^0:0<v(j_{\phi_\beta})\leq x\}\subseteq S_L,\]
where we will eventually take $x$ to be very large.  Note that
for $v\in S_x$ we have $\# \fj{\phi}(L_v)/\phi^0(L_v)$ bounded in terms of $x$, by Lemma~\ref{lem:vjbound}.  In particular, since there are only finitely many $A$-modules of a given (finite) size,
 there is an ideal $\mathfrak{a}\subseteq A$, which depends on $x$, such that $\mathfrak{a}(\fj{\phi}(L_v)/\phi^0(L_v))=\{0\}$ for all  $v\in S_x$.  We claim that, if $x$ is large enough, then $\mu(\phi_\beta, 0, \mathfrak{a})\geq 1/q$ for all but finitely many $\beta\in X(L)$ of bounded inseparable degree.  By the proof of Theorem~\ref{th:szpiroresult}, applied to this family, this gives a uniform bound on $\#\phi^{\mathrm{Tors}}(L)$ for those values of $\beta$.  Of course this bound may be adjusted to accommodate the remaining finitely many fibres in the family.

So we will assume that
\begin{equation}\label{eq:assume}\sum_{v\in S_{x}} j_{\phi, v}\leq \sum_{v\in M_L^0\setminus S_\phi(\mathfrak{a})}j_{\phi, v}<q^{-1}\sum_{v\in M_L^0}j_{\phi, v}.\end{equation}
  Note that, for any place $v\in M_L$, since \[Wv(j_\phi)=v(J_\phi\circ\beta)\in e_v(L/F)\ZZ,\] we certainly have
\[j_{\phi_\beta, v}>0\Longrightarrow Wj_{\phi_\beta, v}=Wv(j_{\phi_\beta})\deg(v)\geq e_v(L/F)\deg(v).\]
Noting that, by Lemma~\ref{lem:insep}, every place of $E$ lifts to a unique place of $L$ with ramification index $e_v(L/E)=[L:E]$, and so we have the weaker statement
\begin{equation}\label{eq:jlower}j_{\phi_\beta, v}>0\Longrightarrow Wj_{\phi_\beta, v}=Wv(j_{\phi_\beta})\deg(v)\geq [L:E]\deg(v).\end{equation}

Now, note that inequality~\eqref{eq:jlower} gives
\begin{equation}\label{eq:smallplaces}\sum_{\substack{v\in M_L^0\\0<v(j_{\phi_\beta})\leq x}}\deg(v)\leq \sum_{\substack{v\in M_L^0\\0<v(j_{\phi_\beta})\leq x}}\frac{W}{[L:E]}j_{\phi_\beta, v}\leq \frac{W}{q[L:E]}\sum_{v\in M_L^0}j_{\phi_\beta, v},\end{equation}
by inequality~\eqref{eq:assume}.

Similarly, inequalities~\eqref{eq:jlower} and~\eqref{eq:assume} also imply
\begin{equation}\label{eq:largeplaces}\sum_{\substack{v\in M_L^0\\v(j_{\phi_\beta})> x}}\deg(v)\leq \sum_{\substack{v\in M_L^0\\v(j_{\phi_\beta})> x}}\frac{1}{x}\cdot\frac{W}{[L:E]}j_{\phi_\beta, v}\leq \frac{W}{x[L:E]}\sum_{v\in M_L^0}j_{\phi_\beta, v}.\end{equation}

Combining inequalities~\eqref{eq:abc},~\eqref{eq:smallplaces}, and~\eqref{eq:largeplaces}, and the fact that $j_{\phi_\beta, v}\geq 0$ for all $v\in M_L$, we deduce that
\begin{eqnarray*}
(\deg_\mathrm{s}(J_\phi)-2)\deg_\mathrm{s}(\beta)&\leq&2g_L-2+\sum_{\substack{v\in M_L\\ j_{\phi_\beta, v}>0}} \deg(v)\\
&\leq& 2g_L-2+\sum_{\substack{v\in M_L^0\\0<v(j_{\phi_\beta})\leq x}}\deg(v)+\sum_{\substack{v\in M_L^0\\ v(j_{\phi_\beta})>x}}\deg(v)\\
&& + \sum_{v\in M_L^\infty} \deg(v)\\
&\leq &2g_L-2+ \left(\frac{1}{q}+\frac{1}{x}\right)\frac{W}{[L:E]} \sum_{v\in M_L^0}j_{\phi_\beta, v}+[L:K]\\
&\leq & 2g_L-2 +\left(\frac{1}{q}+\frac{1}{x}\right)\frac{W}{[L:E]} h_L(j_\phi)+[L:K]\\
&=&2g_L-2+\left(\frac{1}{q}+\frac{1}{x}\right)\frac{1}{\deg_\mathrm{i}(J_\phi\circ\beta)}\deg(J_\phi\circ\beta)+[L:K]\\
&=&2g_L-2+\left(\frac{1}{q}+\frac{1}{x}\right)\deg_\mathrm{s}(\beta)\deg_s(J_\phi)+[L:K].
\end{eqnarray*}
Re-arranging the above, with the assumptions of Theorem~\ref{th:simplefams}, we have
\begin{equation}\label{eq:degsbound}\left(\deg_\mathrm{s}(J_\phi)\left(1-\frac{1}{q}-\frac{1}{x}\right)-2\right)\deg_\mathrm{s}(\beta)\leq 2g_L-2+[L:K]\leq 0,\end{equation}
as $L$ is rational, and $[L:K]\leq 2$.
As long as \[q> \deg_\mathrm{s}(J_\phi)/( \deg_\mathrm{s}(J_\phi)-2),\] which is certainly the case if $\deg_\mathrm{s}(J_\phi)\geq m_q$, the left-hand-side of \eqref{eq:degsbound} is positive once $x$ is larger than a quantity depending only on $\deg_\mathrm{s}(J_\phi)$ and $q$.

More generally, under the assumptions of Theorem~\ref{th:simplefams2}, the inequality~\eqref{eq:degsbound} bounds $h_L(\beta)=\deg(\beta)$ in terms of $g_L$, $[L:K]$, and $\deg_\mathrm{i}(\beta)$.  We may thus relax our hypotheses that $L$ be a rational, at most quadratic, extension of $K$, if we are content with a bound on $\# \phi_\beta^{\mathrm{Tors}}(L)$ which depends on $\deg_\mathrm{i}(\beta)$.

\end{proof}

\begin{remark}
It is also worth noting that the restriction $[L:K]\leq 2$ in Theorem~\ref{th:simplefams} is only necessary because we might have $j_{\phi_\beta, v}>0$ for infinite places $v\in M_L$.  If we rule out this possibility, by considering only fibres over $\beta\in L$ such that $S_{L, \beta}$ is disjoint from $M_L^\infty$, then we obtain a uniform bound for any rational extension $L/K$.
\end{remark}

%%%%%%%%%%%%%%%%%%%%%%%%%%%%%%%%%%%%%%%%%%%%%%%%%
%%%%%%%%%%%%%%%%%%%%%%%%%%%%%%%%%%%%%%%%%%%%%%%%%
%%%%%%%%%%%%%%%%%%%%%%%%%%%%%%%%%%%%%%%%%%%%%%%%%
%%%%%%%%%%%%%%%%%%%%%%%%%%%%%%%%%%%%%%%%%%%%%%%%%
%%%%%%%%%%%%%%%%%%%%%%%%%%%%%%%%%%%%%%%%%%%%%%%%%
%%%%%%%%%%%%%%%%%%%%%%%%%%%%%%%%%%%%%%%%%%%%%%%%%

\section{Tate uniformization and the filled Julia set}
\label{sec:uniformization}

For the entirety of this section, we will assume that $A=\FF_q[T]$.

In Section~\ref{sec:localfields} we defined, given a Drinfeld module $\phi$ over a complete, algebraically closed $A$-field $\CC_v$, two submodules $\phi^0(\CC_v)\subseteq\fj{\phi}(\CC_v)$ of $\phi(\CC_v)$, with the property that $\overline{\phi^0(\CC_v)}$ is the connected component of $\overline{\fj{\phi}(\CC_v)}$ containing the identity, where the closure here is relative to the topology on $\Aberk{\CC_v}$.  As we are motivated in part by the analogy with the elliptic case, we will recall a basic fact about the component group $E/E^0$ of an elliptic curve over an extension of $\QQ_p$.  Since every elliptic curve, after some finite extension of the base, and a change of coordinates, has split multiplicative reduction, we will consider that case.  We refer the reader to \cite[Ch.~IV, V]{ataec} for more details.
\begin{theorem}[Kodaira, N\'{e}ron, Tate]
Let $L/\QQ_p$ be a finite extension, and let $E/L$ be an elliptic curve with split multiplicative reduction.  Then if $E^0$ is the connected component of the N\'{e}ron model of $E$ containing the identity, we have
\[E(\CC_p)/E^0(\CC_p)\cong \QQ/\ZZ\]
and
\[E(L)/E^0(L)\cong \ZZ/N\ZZ,\]
for some $N\in \ZZ$.
\end{theorem}

Indeed, the value $N$ in the theorem is $-v(j_E)$ but it is, in many cases, enough to know simply that $E(L)/E^0(L)$ is finite, and is trivial in the case of good reduction.  The aim of this section is to prove an analogous result for filled Julia sets of Drinfeld modules at finite places.  The natural replacements for $\QQ$ and $\ZZ$ are, of course, $K$ and $A$. 
We know already that if $L/K_v$ is a finite extension, then $\fj{\phi}(L)/\phi^0(L)$ is a finite $A$-module, trivial if $\phi$ has (potentially) good reduction.  Recall that the Drinfeld module $\phi/L$ has \emph{stable reduction of rank $s$} if the coefficients of $\phi_a(x)$ are integral, for every $a\in A$, and if the reduction of $\phi$ modulo the maximal ideal induces a Drinfeld module on the residue field, and this Drinfeld module has rank $s$.  We note that this is the same as requiring that for \emph{some} non-constant $T\in A$,  $\phi_T(x)$ has integral coefficients, the coefficient corresponding to $s\deg(T)$ is a unit, and the coefficients corresponding to higher indices are not.  We will say that $\phi/L$ has \emph{potentially stable} reduction of rank $s$ if there is a finite extension of $L$, and a change of variables, after which $\phi$ has stable reduction of rank $s$.  Every Drinfeld module has potentially stable reduction.  
\begin{theorem}\label{th:KN-Drinfeld}
Let $L/K_v$ be a finite extension, where $v\in M_K^0$, and let $\phi/L$ be a Drinfeld module of rank $r$, with stable reduction of rank $r-s$, and let $\CC_v$ be a completion of an algebraic closure of $L$.   Then
\[\fj{\phi}(\CC_v)/\phi^0(\CC_v)\cong (K/A)^{s},\]
as $A$-modules,
and there exist non-zero ideals $a_1, ..., a_s\in A$ such that
\[\fj{\phi}(L)/\phi^0(L)\cong A/a_1A \cdots\oplus A/a_sA.\]
\end{theorem}

Note that, in the elliptic curve case, we know that $E(L)/E^0(L)\cong \ZZ/N\ZZ$, where $N=v(j_{E})$ (in the semi-stable case).  In the Drinfeld module context we do not obtain as strong a bound, although we do have $\deg(a_i)$ bounded in terms of $v(j_\phi$, by Lemma~\ref{lem:vjbound} above.

Much as the characterization of the component group of the special fibre of the N\'{e}ron model of an elliptic curve over a non-archimedean field can be obtained from the theory of Tate uniformization, our characterization in the case of Drinfeld modules makes use of the analogous theory.

Let $\psi/\CC_v$ be a Drinfeld module or rank $r_1$, and let $\Lambda\subseteq \psi(\CC_v)$ be a lattice of rank $s$.  That is, let $\Lambda\subseteq \psi(\CC_v)$ be a projective $A$-submodule of rank $s$ such that $\Lambda\cap D$ is finite, for any disk $D$.  We define an exponential map associated to $\Lambda$ by
\[e_\Lambda(x)=x\prod_{\omega\in\Lambda}\left(1-\frac{x}{\omega}\right).\]

The following theorem is due to Drinfeld \cite[Proposition 7.2]{drinfeld}.
\begin{theorem}\label{th:tate}
There is a one-to-one correspondence between isomorphism classes of Drinfeld modules $\phi/\CC_v$ of rank $r$ and potentially stable reduction of rank $r_1\leq r$  and isomorphism classes of pairs $(\psi, \Lambda)$ as above, where $\psi/\CC_v$ has rank $r_1$ and potentially good reduction, and $\Lambda$ has rank $s=r-r_1$.
\end{theorem}

One proves this by constructing a Drinfeld module $\phi/\CC_v$ such that
\[e_\Lambda(\psi_a(x))=\phi_a(e_\Lambda(x))\]
for all $x\in \CC_v$ and all $a\in A$, and showing that every $\phi/\CC_v$ with potentially stable reduction of rank $r_1$ is obtained in this way.

We note that if $\Lambda\subseteq\psi(\CC_v)$ is a lattice, then $\Lambda\cap\Ocal_v$ is finite, since $\Ocal_v$ is the disk $D(0, 1)$.  On the other hand, if $\psi$ has good reduction, then $\Lambda\cap\Ocal_v$ is a submodule of $\psi(\CC_v)$, and so must in fact be trivial (since $\Lambda$ is projective and $A$ is infinite).  
  The following lemma has its origins in Lemma~4.2 of \cite{taguchi}.

\begin{lemma}\label{lem:rigid}
Let $\psi/\CC_v$ have rank $r_1$ and good reduction, and let $\Lambda\subseteq\CC_v$ be a $\psi$-lattice of rank $s$.  Then there exists a basis $\omega_1, ..., \omega_s$ for $\Lambda$ such that for all $a_1, ..., a_s\in A$,
\begin{equation}\label{eq:rigid}\log\left|\sum_{i=1}^s\psi_{a_i}(\omega_i)\right|=\max\left\{|a_1|_\infty^{r_1}\log|\omega_1|,..., |a_s|^{r_1}_\infty\log|\omega_s|\right\}.\end{equation}
\end{lemma}

\begin{proof}
Let $\Lambda_0=\{0\}$, and
for each $i\geq 1$, choose $\omega_{i}\in \Lambda\setminus\Lambda_{i-1}$ of minimal absolute value (which exists, since $\Lambda$ is discrete).
Let $\Lambda_i\subseteq\Lambda$ be the submodule generated by $\omega_1, ..., \omega_{i}$.  We will prove, by induction, that $\dim_K(\Lambda_i\otimes_A K)=\dim(\Lambda_{i-1}\otimes_A K)+1$, and that
the collection $\omega_1, ..., \omega_{i}$ satisfies \eqref{eq:rigid}.  Note, in particular, that the first assertion shows that the construction terminates with $\Lambda_s=\Lambda$.

First, note that since $\omega_1\neq\{0\}$, we have $\dim(\Lambda_1\otimes_A K)=1$.  Also, since $|\omega_1|>1$ but the coefficients of $\psi_a$ are integral, for any $a\in A$, we have
\[\log|\psi_a(\omega_1)|=|a|_\infty^{r_1}\log|\omega_1|.\]
This proves~\eqref{eq:rigid} for $\Lambda_1$.  

Now, suppose that the assertions hold for all $i< k$, consider $\omega_{k}\in \Lambda\setminus\Lambda_{k-1}$ of minimal absolute value, and let $\Lambda_k$ be the submodule of $\Lambda$ generated by $\Lambda_{k-1}$ and $\omega_k$.  First suppose that $\dim(\Lambda_k\otimes_A K)=\dim(\Lambda_{k-1}\otimes_A K)$.  Then there exist $a_1, ..., a_{k-1}, b\in A$  such that
\[\sum_{i=1}^{k-1} \psi_{a_i}(\omega_i)=\psi_b(\omega_{k}).\]
Now, for each $i$, write $a_i=bc_i+d_i$ with $\deg(d_i)<\deg(b)$, allowing $c_i=0$.  We then have
\[\sum_{i=1}^{k-1} \psi_{d_i}(\omega_i)=\psi_b\left(\omega_{k}-\sum_{i=1}^{k-1}\psi_{c_i}(\omega_i)\right).\]
In particular,
\begin{eqnarray*}
\log\left|\omega_{k}-\sum_{i=1}^{k-1}\psi_{c_i}(\omega_i)\right|&=&|b|_\infty^{-r_1}\log\left|\sum_{i=1}^{k-1} \psi_{d_i}(\omega_i)\right|\\
&=&|b|_\infty^{-r_1}\max\left\{|d_1|_\infty^{r_1}\log|\omega_1|, ..., |d_{k-1}|_\infty^{r_1}\log|\omega_{k-1}|\right\}\\
&<&\max\left\{\log|\omega_1|, ..., \log|\omega_{k-1}|\right\}\leq \log|\omega_k|,
\end{eqnarray*}
since $|b|_\infty>|d_i|_\infty$, for all $i$, and since $\omega_{i}$ had minimal absolute value in $\Lambda\setminus\Lambda_{i-1}$.  This contradicts the minimal size of $\omega_k$ amongst elements of $\Lambda\setminus\Lambda_{k-1}$.
So we must have $\dim(\Lambda_k\otimes_A K)=\dim(\Lambda_{k-1}\otimes_A K)+1$.

 Now, we wish to show that~\eqref{eq:rigid} holds for $\Lambda_k$, that is, that for all $a_1, ..., a_{k}\in A$  we have
\[\left|\sum_{i=1}^{k}\phi_{a_i}(\omega_i)\right|=\max\left\{|a_1|_\infty^{r_1}\log|\omega_1|,..., |a_k|^{r_1}_\infty\log|\omega_k|\right\}.\]
If this fails for some choice of $a_1, ..., a_{k}\in A$, then we must have $a_{k}\neq 0$, lest we contradict the induction hypothesis.  Similarly, we must have
\[|a_{k}|_\infty^{r_1}\log\left|\omega_{k}\right|=\max\left\{|a_1|_\infty^{r_1}\log|\omega_1|, ..., |a_{k-1}|^{r_1}_\infty\log|\omega_{k-1}|\right\},\]
or else the equality holds from the ultrametric inequality.  If \eqref{eq:rigid} fails, we have
\[\log\left|\sum_{i=1}^{k}\phi_{a_i}(\omega_i)\right|<|a_{k}|_\infty^{r_1}\log|\omega_{k}|.\]
Now, for each $1\leq i\leq k-1$, let $a_i=a_{k}c_i+d_i$ with $\deg(d_i)<\deg(a_{k})$.  Then we have
\[\sum_{l=1}^{k}\phi_{a_i}(\omega_i)=\sum_{l=1}^{k-1}\phi_{d_i}(\omega_i)+\phi_{a_{k}}\left(\omega_{k}+\sum_{i=1}^{k-1}\phi_{c_i}(\omega_l)\right).\]
Since
\[\log\left|\sum_{i=1}^{k-1}\phi_{d_i}(\omega_i)\right|=\max\left\{|d_1|_\infty^{r_1}\log|\omega_1|, ..., |d_{k-1}|_\infty^{r_1}, |\xi|\right\}<|a_{k}|_\infty^{r_1}\log|\omega_{k}|,\]
it must also be the case that
\[\log\left|\phi_{a_{k}}\left(\omega_{k}+\sum_{i=1}^{k-1}\phi_{c_i}(\omega_i)\right)\right|<|a_{k}|_\infty^{r_1}\log|\omega_{k}|,\]
as well.  But this implies
\[\left|\omega_{k}+\sum_{i=1}^{k-1}\phi_{c_i}(\omega_i)\right|<|\omega_{k}|,\]
contradicting the minimality of $|\omega_k|$ amongst members of $\Lambda\setminus\Lambda_{k-1}$.  By induction, the lemma holds.
\end{proof}

The proof of Theorem~\ref{th:KN-Drinfeld} comes in three lemmas.

\begin{lemma}\label{lem:finiteness}
Under the hypotheses of Theorem~\ref{th:KN-Drinfeld}, the $A$-module $\fj{\phi}(L)/\phi^0(L)$ is finite.
\end{lemma}

\begin{proof}
Lemma~\ref{lem:vjbound} gives a bound on the size of the component module, but was proven under the assumption that $L$ is  a finite extension of the completion of $K$ at some place.  It is not hard to modify the proof, or to simply note that $\fj{\phi}(L)$ is contained in a disk of finite radius, while $\phi^0(L)$ contains a disk of positive radius.  Since $L$ is the fraction field of a discrete valuation ring, this is enough to ensure that $\fj{\phi}(L)/\phi^0(L)$ is finite.
\end{proof}

Note that, from the proof of Lemma~\ref{lem:finiteness}, we may work out an explicit upper bound on the size of $\fj{\phi}(L)/\phi^0(L)$ in terms of various data relating to $\phi/L$.  In Section~\ref{sec:global} we will work out a sharper bound, which uses more information than the trivial estimate on the smallest disk containing $\fj{\phi}(L)$.

We may now establish the characterization of $\fj{\phi}(\CC_v)/\phi^0(\CC_v)$.

\begin{lemma}\label{lem:isom}
Let $\phi/L$ have potentially stable reduction of rank $r-s$.  
Then
\[\fj{\phi}(\CC_v)/\phi^0(\CC_v)\cong (K/A)^{s}\]
as $A$-modules.
\end{lemma}

\begin{proof}  Again, without loss of generality, we suppose that $\phi/L$ has stable reduction.
We first note that, by definition, $\overline{L}$ is dense in $\CC_v$.  So if $z\in\fj{\phi}(\CC_v)$, then there is some $w\in \overline{L}$ with $|z-w|_v<1$.  But then $w\in \fj{\phi}(\overline{L})$, and $z$ and $w$ represent the same coset in $\fj{\phi}(\CC_v)/\phi^0(\CC_v)$.  This shows that the natural map $\fj{\phi}(\overline{L})\to\fj{\phi}(\CC_v)/\phi^0(\CC_v)$ is surjective.  For any subfield $F\subseteq \CC_v$, we have $\phi^0(F)=F\cap D(0, 1)$, and so the  kernel of this map is $\phi^0(\overline{L})$, ensuring
\[\fj{\phi}(\CC_v)/\phi^0(\CC_v)\cong \fj{\phi}(\overline{L})/\phi^0(\overline{L}).\]
Now, by Lemma~\ref{lem:finiteness}, we know that $\fj{\phi}(L)/\phi^0(L)$ is a finite $A$-module, and this applies as well to any finite extension of $L$.  In particular, it follows that $\fj{\phi}(\CC_v)/\phi^0(\CC_v)$ is a direct limit of finite modules, and so is a torsion module.

Now, let $(\psi, \Lambda)$ the Tate datum associated to $\phi/\CC_v$ by Theorem~\ref{th:tate}, so that $\psi/\CC_v$ has good reduction, and $e\circ \psi=\phi\circ e$, for
\begin{equation}\label{eq:exponential}e(z)=z\prod_{\substack{\omega\in \Lambda\\ \omega\neq 0}}\left(1-\frac{z}{\omega}\right).\end{equation}
Note that $e:\CC_v\to\CC_v$ is a surjective homomorphism with kernel $\Lambda$.  Now, let
\[H_\Lambda=\left\{z\in\psi(\CC_v):\psi_a(z)\in \Lambda\text{ for some }a\in A\setminus\{0\}\right\},\]
a submodule of $\psi(\CC_v)$, 
and let $\Ocal_v\subseteq \CC_v$ denote the ring of integral elements of $\CC_v$. 
We will show that $e$ induces an isomorphism 
\[(H_\Lambda+\Ocal_v)/(\Lambda+\Ocal_v)\to \fj{\phi}(\CC_v)/\phi^0(\CC_v),\]
and then study the former $A$-module.

Note that for any $z\in \CC_v$ we have
\[|e(z)|=|z|\prod_{\substack{\omega\in\Lambda\\\omega\neq 0}}\left|1-\frac{z}{\omega}\right|\leq |z|,\]
and so
  $e(\Lambda+\Ocal_v)\subseteq \Ocal_v=\phi^0(\CC_v)$.  On the other hand,  suppose that $|z|\leq |z-\omega|$ for all $\omega\in\Lambda$.  
 Then
 \[|e(z)|=|z|\prod_{\substack{|\omega|\leq |z|\\\omega\neq 0}}\left|1-\frac{z}{\omega}\right|=|z|^{\#\{\omega\in\Lambda:|\omega|\leq |z|\}}.\]
 Since every coset $z+\Lambda$ has a representative with this property, if $z\not\in \Lambda+\Ocal_v$ we have $|e(z)|>1$.
In other words, $e(\CC_v\setminus(\Lambda+\Ocal_v))\subseteq \CC_v\setminus\Ocal_v$.  But $e$ is surjective, and so we must have $e(\Lambda+\Ocal_v)=\Ocal_v$.

Since $\psi$ has good reduction, $\psi_a(\Ocal_v)\subseteq \Ocal_v$ for all $a\in A$.  As a consequence, if $z\in H_\Lambda+\Ocal_v$, then for some non-zero $a\in A$, we have $\psi_a(z)\in \Lambda+\Ocal_v$, and hence
\[\phi_a(e(z))=e(\psi_a(z))\in e(\Lambda+\Ocal_v)=\phi^0(\CC_v).\]   But $\phi_a(x)\in\fj{\phi}(\CC_v)$ for any non-zero $a\in A$ implies  $x\in \fj{\phi}(\CC_v)$.  To recapitulate, we have shown that
\[e(H_\Lambda+\Ocal_v)\subseteq \fj{\phi}(\CC_v).\]
In fact, this inclusion is an equality.
Suppose that $x\in \fj{\phi}(\CC_v)$.  Since $\fj{\phi}(\CC_v)/\phi^0(\CC_v)$ is torsion, there is some non-zero $a\in A$ such that $\phi_a(x)\in \phi^0(\CC_v)=e(\Lambda+\Ocal_v)$. In particular, $\phi_a(x)=e(z)$ for some $z\in \Lambda+\Ocal_v$.  Now, choose $w\in\CC_v$ with $\psi_a(w)=z$.  Then $\phi_a(x)-\phi_a(e(w))=0$, and so $x-e(w)\in\phi[a]$.  But $\phi[a]$ is the image under $e$ of $\frac{1}{a}\Lambda+\psi[a]\subseteq H_\Lambda+\Ocal_v$, since $\psi$ has good reduction.  So $x-e(w)\in e(H_\Lambda+\Ocal_v)$, whence $x\in e(H_\Lambda+\Ocal_v)$. 

We now have a surjective homomorphism
\[e:H_\Lambda+\Ocal_v\to\fj{\phi}(\CC_v)/\phi^0(\CC_v)\]
and, since $e(\Lambda+\Ocal_v)=\Ocal_v=\phi^0(\CC_v)$, we see that the kernel of this map is precisely $\Lambda+\Ocal_v$.  In other words, we have established the claim that
\[(H_\Lambda+\Ocal_v)/(\Lambda+\Ocal_v)\to \fj{\phi}(\CC_v)/\phi^0(\CC_v).\]
Note that since $\psi(\Ocal_v)\subseteq \Ocal_v$, and since $\Lambda\cap \Ocal_v=\{0\}$, we must have $H_\Lambda\cap \Ocal_v=\psi^{\mathrm{Tors}}$.
It now suffices to describe
\begin{eqnarray*}
(H_\Lambda+\Ocal_v)/(\Lambda+\Ocal_v)&=& (H_\Lambda+\Lambda+\Ocal_v)/(\Lambda+\Ocal_v)\\
&=& H_\Lambda/(H_\Lambda\cap (\Lambda+\Ocal_v))\\
&=&H_\Lambda/(\Lambda+\psi^{\mathrm{Tors}}),
\end{eqnarray*}
since $\Lambda\subseteq H$ and $H\cap\Ocal_v=\psi^{\mathrm{Tors}}$.

If $\omega_1, ..., \omega_s$ is a basis for $\Lambda$ over $\Lambda^{\mathrm{Tors}}$, as in Lemma~\ref{lem:rigid}, we define a map
\[\Psi:K^s\to H_\Lambda/(\Lambda+\psi^{\mathrm{Tors}})\]
by setting
\[\Psi(x_1, ..., x_s)=y+(\Lambda+\psi^{\mathrm{Tors}})\text{ if and only if }\sum \psi_{x_id}(\omega_i)=\psi_d(y)\]
whenever $d\in A$ satisfies $x_id\in A$ for all $i$.  First we must show that this is well-defined, so fix $x_1, ..., x_s\in K$ and $d_1, d_2\in A$ such that $x_id_j\in A$ for all $i$ and $j$, and suppose that
$\sum \psi_{d_jx_i}(\omega_i)=\psi_{d_j}(y_j)$ for $j=1, 2$.  Then we have
\[\psi_{d_2d_1}(y_1)=\sum\psi_{d_2d_1x_i}(\omega_i)=\psi_{d_1d_2}(y_2),\] and hence $\psi_{d_2d_2}(y_1-y_2)=0$, whereupon $y_1-y_2\in\psi^{\mathrm{Tors}}\subseteq \Lambda+\psi^{\mathrm{Tors}}(\CC_v)$

The map $\Psi$ is also clearly surjective, by the definition of $H$.  It remains to compute $\ker(\Psi)$.  But if \[\Psi(x_1, ..., x_s)=0+ \left(\Lambda+\psi^{\mathrm{Tors}}(\CC_v)\right)\] then we have some $d\in A$ such that $dx_i\in A$ for all $i$ and some $\omega+\zeta\in \Lambda+\psi^{\mathrm{Tors}}(\CC_v)$ such that \[\sum\psi_{dx_i}(\omega_i)= \psi_d(\omega+\zeta).\]  If $\omega=\sum \psi_{a_i}(\omega_i)+\xi'$, for $a_i\in A$ and $\xi'\in\psi^{\mathrm{Tors}}(\CC_v)$, then we have
\[\sum\psi_{d(a_i-x_i)}(\omega_i)=\xi-\psi_d(\zeta+\xi')\in\psi^{\mathrm{Tors}}.\]
If follows that both sides vanish, and hence $x_i=a_i\in A$ for all $i$.  In other words, $\ker(\Psi)=A^s\subseteq K^s$, proving the lemma.
\end{proof}

\begin{lemma}
Let $L/K_v$ be finite.  Then there exist $a_i\in A$ such that
\[\fj{\phi}(L)/\phi^0(L)\cong \bigoplus_{i=1}^s A/a_iA\]
as an $A$-module.
\end{lemma}

\begin{proof}
This follows immediately from the fact that $\fj{\phi}(L)/\phi^0(L)$ is isomorphic to a finite submodule of $(K/A)^s$.
\end{proof}

%We conclude this section by noting, as alluded to in Lemma~\ref{lem:rigid}, that the local height $\lambda_\phi$ on $\fj{\phi}(\CC_v)\phi^0(\CC_v)$ admits a particularly simple description.
%\begin{theorem}
%Let $\phi/\CC_v$ be a Drinfeld module or rank $r$ with potentially stable reduction of rank $r-s$, and let $u:(K/A)^s\to \fj{\phi}(\CC_v)/\phi^0(\CC_v)$ be the isomorphism defined in Lemma~\ref{lem:isom}.  Then there is a norm $\|\cdot\|$ on $K^s$, compatible with the absolute value $|\cdot|_\infty$ on $K$, such that if $\vec{x}\in K^s$ is in the fundamental domain $\{\vec{x}\in K^2:|x_i|_\infty<1\}$ then
%\[\lambda_\phi\circ u(\vec{x})=-\log\|\vec{x}\|+c(\phi).\]
%\end{theorem}

\section{Elliptic curves}
\label{sec:elliptic}

We conclude with some more details on the analogy between Conjecture~\ref{question} and the case of elliptic curves.
Let $F$ be number field, let $E/F$ be an elliptic curve, and let $j_{E, v}=\log^+|j_E|_v$, for every normalized valuation $v\in M_F$.  Then $j_{\phi, v}$ is analogous to $j_{E, v}$ in that both are non-negative, vanish just in the case of potential good reduction (for finite places), and sum over all places to give the height of the moduli representative of the isomorphism class.

Now, for each ideal  $\mathfrak{a}\subseteq\ZZ$, let
\[S_E(\mathfrak{a})=\{v\in M_F^0:\mathfrak{a}(E(F_v)/E^0(F_v))\neq \{0\}\},\]
and let
\[\mu(E, N, \mathfrak{a})=\max_{\substack{S\subseteq M_F^0\\ \# S=N}}\frac{\sum_{v\in M_F^0\setminus (S_E(\mathfrak{a})\cup S)}[F_v:\QQ_v]j_{E, v}}{\sum_{v\in M_K^0\setminus S}[F_v:\QQ_v]j_{E, v}}.\]
We have, as above, $0\leq \mu(E, N, \mathfrak{a})\leq 1$, with equality on the upper bound if $N$ is sufficiently large, or if $\mathfrak{a}$ is sufficiently divisible.

The following argument comes directly from \cite{hindry-silv}.
\begin{proof}[Proof of Theorem~\ref{th:elliptic}]
Again, for simplicity, we will define $\deg(v)$, for every $v\in M_F^0$, such that $\log|x|_v=-v(x)\deg(v)$.
We have, by definition,
\[\log|\operatorname{Norm}_{F/\QQ}\Delta_E|=\sigma(E/F)\log|\operatorname{Norm}_{F/\QQ}f_E|.\]
Since $E$ is semi-stable, for each valuation $v$ we have $v(f_E)=1$, if $E$ has bad reduction at $v$, and $v(f_E)=0$ if $E$ has good reduction.   Also, $E(F_v)/E^0(F_v)$ is cyclic of order $v(j_E)$, and so if $\mathfrak{a}=(n!)\ZZ$, then we have $v(j_E)>n$ for all $v\in S_E(\mathfrak{a})$.  It follows that
\begin{eqnarray*}
\sum_{v\in M_F^0}[F_v:\QQ_v]j_{E, v}&=&\log|\operatorname{Norm}_{F/\QQ}\Delta_E|\\
&=&\sigma(E/F)\log|\operatorname{Norm}_{F/\QQ}f_E|\\
&=&\sigma(E/F)\sum_{v\in M_F^0\setminus S_E(\mathfrak{a})} [F_v:\QQ_v]\deg(v)\\&&+\sigma(E/F)\sum_{v\in S_E(\mathfrak{a})}[F_v:\QQ_v] \deg(v)\\
&\leq &\sigma(E/F)\sum_{v\in M_F^0\setminus S_E(\mathfrak{a})}[F_v:\QQ_v]j_{E, v}\\&&+\sigma(E/F)\frac{1}{n}\sum_{v\in S_E(\mathfrak{a})}[F_v:\QQ_v]j_{E, v}\\
&=& \sigma(E/F) \sum_{v\in M_F^0\setminus S_E(\mathfrak{a})} [F_v:\QQ_v]j_{E, v}\\&&+\frac{\sigma(E/F)}{n}\left(\sum_{v\in M_F^0}[F_v:\QQ_v]j_{E, v}-\sum_{v\in M_F^0\setminus S_E(\mathfrak{a})} [F_v:\QQ_v]j_{E, v}\right),
\end{eqnarray*}
and so
\[\left(1-\frac{\sigma(E/F)}{n}\right)\sum_{v\in M_F^0}[F_v:\QQ_v]j_{E, v}\leq \sigma(E/F)\left(1-\frac{1}{n}\right)\sum_{v\in M_F^0\setminus S_E(\mathfrak{a})}[F_v:\QQ_v]j_{E, v}.\]
\end{proof}

%%%%%%%%%%%%%%%%%%%%%%%%%%%%%%%%%%%%%%%%%
%%%%%%%%%%%%%%%%%%%%%%%%%%%%%%%%%%%%%%%%%
%%%%%%%%%%%%%%%%%%%%%%%%%%%%%%%%%%%%%%%%%
%%%%%%%%%%%%%%%%%%%%%%%%%%%%%%%%%%%%%%%%%
%%%%%%%%%%%%%%%%%%%%%%%%%%%%%%%%%%%%%%%%%
%%%%%%%%%%%%%%%%%%%%%%%%%%%%%%%%%%%%%%%%%
%%%%%%%%%%%%%%%%%%%%%%%%%%%%%%%%%%%%%%%%%
%%%%%%%%%%%%%%%%%%%%%%%%%%%%%%%%%%%%%%%%%
%%%%%%%%%%%%%%%%%%%%%%%%%%%%%%%%%%%%%%%%%
%%%%%%%%%%%%%%%%%%%%%%%%%%%%%%%%%%%%%%%%%
%%%%%%%%%%%%%%%%%%%%%%%%%%%%%%%%%%%%%%%%%
%%%%%%%%%%%%%%%%%%%%%%%%%%%%%%%%%%%%%%%%%
%%%%%%%%%%%%%%%%%%%%%%%%%%%%%%%%%%%%%%%%%
%%%%%%%%%%%%%%%%%%%%%%%%%%%%%%%%%%%%%%%%%
%%%%%%%%%%%%%%%%%%%%%%%%%%%%%%%%%%%%%%%%%
%%%%%%%%%%%%%%%%%%%%%%%%%%%%%%%%%%%%%%%%%

\end{document}